\documentclass[11pt]{amsart}
\usepackage{amsmath,amssymb,amsfonts,amsthm, amscd,indentfirst}
\usepackage{amsmath,latexsym,amssymb,amsmath,
	amscd,amsthm,amsxtra}
\usepackage[hidelinks]{hyperref}
\usepackage{url}
\usepackage{color}
\usepackage{pdfpages}
\usepackage{graphics}
\usepackage{enumitem}  
\usepackage{amssymb,amsmath}
\usepackage{enumerate,amsmath,amssymb, mathrsfs}
\usepackage{latexsym}
\newtheorem{theorem}{Theorem}[section]

\newtheorem{lemma}{Lemma}[section]
\newtheorem{corollary}{Corollary}[section]
\newtheorem{definition}{Definition}[section]

\newtheorem{remark}{\textbf{Remark}}[section]
\begin{document}

\noindent This is a pre-print version of the article \textit{Sharp Estimates for the Principal Eigenvalue of the $p-$Operator} which appeared in Calculus of Variations and Partial Differential Equations 57:49 on April 18th 2018, 
{\centering \url{https://link.springer.com/article/10.1007/s00526-018-1331-0}}.
\\
\title{Sharp Estimates for the Principal Eigenvalue of the $p-$Operator}
\author{Thomas Koerber}
\address{	Albert-Ludwigs-Universit\"at Freiburg, Mathematisches Institut, Eckerstr. 1, D-79104 Freiburg, Germany \\
	Tel.: +49-761-203-5614}
	\email{thomas.koerber@math.uni-freiburg.de}

\begin{abstract}
 Given an elliptic diffusion operator $L$ defined on a compact and connected manifold (possibly with a convex boundary in a suitable sense) with an $L$-invariant measure $m$, we introduce the non-linear $p-$operator $L_p$, generalizing the notion of the $p-$Laplacian. Using techniques of the intrinsic $\Gamma_2$-calculus, we prove the sharp estimate $\lambda\geq (p-1)\pi_p^p/D^p$ for the principal eigenvalue of $L_p$ with Neumann boundary conditions under the assumption that $L$ satisfies the curvature-dimension condition BE$(0,N)$ for some $N\in[1,\infty)$. Here, $D$ denotes the intrinsic diameter of $L$. Equality holds if and only if $L$ satisfies BE$(0,1)$. We also derive the lower bound $\pi^2/D^2+a/2$ for the real part of the principal eigenvalue of a non-symmetric operator $L=\Delta_g+X\cdot\nabla$ satisfying $\operatorname{BE}(a,\infty)$.   
\end{abstract}
\date{}
\maketitle

\medskip

\section{Introduction}
\label{introduction}
Estimating the principal eigenvalue of second-order operators is an important topic for various reasons. In numerical analysis, the principal eigenvalue corresponds to the convergence rate of numerical schemes, and good estimates can lead to an optimization of these schemes. On the other hand, the first eigenvalue often corresponds to the optimal constant of Poincar\'e-type inequalities. While a good understanding of such constants is important for numerical purposes, it can be useful for purely mathematical reasons too (see for instance the solution of the Yamabe problem \cite{yamabe}). Finally, in quantum mechanics, the principal eigenvalue describes the energy of a particle in the ground state (see for instance \cite{cohen}), whereas in thermodynamics, it gives  a lower bound on the decay rate of certain heat flows (see \cite{widder}).\\
 \indent Given its physical and mathematical importance, the first eigenvalue of the Laplace-Beltrami operator on compact Riemannian manifolds with Neumann boundary conditions has been studied in various articles: one of the first remarkable results was \cite{payne} in 1960, where Payne and Weinberger showed the sharp estimate $\lambda\geq\pi^2/D^2$ for the first eigenvalue of the Laplacian defined on a convex and bounded open subset of $\mathbb{R}^n$ with diameter $D$. In 1970, by relating the principal eigenvalue of $L$ to the isoperimetric constant of a Riemannian manifold, Cheeger showed in \cite{cheeger} that the principal eigenvalue can be estimated by a quantity depending only on the diameter, the Ricci-curvature and the dimension.  In 1980, assuming non-negative Ricci-curvature, Li and Yau (\cite{li}) showed the estimate $\lambda\geq\pi^2/4D^2$ for any compact Riemannian manifold, possibly with  convex boundary, using a gradient estimate technique. In 1984, Zhong and Yang finally derived the sharp estimate 
 $
 \lambda\geq\pi^2/D^2
 $
 using a barrier argument (see \cite{Yang}). Afterwards, Chen and Wang (\cite{chen}), and also Kr\"oger (\cite{kroger}), recovered this result independently by comparing the principal eigenvalue to the first eigenvalue of a one-dimensional model space. While Chen and Wang used a variational formula, Kr\"oger used a gradient comparison technique. Meanwhile, more general linear elliptic operators have been studied, and in \cite{emery}, Bakry and Emery introduced intrinsic objects like a generalized metric $\Gamma$, a Hessian $H$, a diameter $D$, and Ricci curvature $R$ related to a so-called diffusion operator $L$. They also introduced a curvature dimension condition solely depending on $L$, which is now known as $\operatorname{BE}(\kappa,N)$, where $\kappa$ is a lower bound for the curvature and $N$ an upper bound for the dimension. In general, $N$ does not coincide with the topological dimension of the manifold. In the year 2000, Bakry and Qian used techniques similar to \cite{kroger} to obtain the sharp estimate $\lambda\geq\pi^2/D^2$ for the principal eigenvalue of $L$ assuming the condition $\operatorname{BE}(0,N)$ for some $N\in[1,\infty)$ and ellipticity (see (\cite{Bakry1}). Thereby, they also obtained sharp results for positive or negative lower bounds on the Ricci-curvature. It is worth remarking that positive bounds are a lot easier to deal with and that Lichnerowicz already obtained the sharp estimate $\lambda\geq \kappa N$ in 1958, where $\kappa>0$ is a lower bound for the Ricci-curvature and $N$ the dimension of the manifold (see \cite{lich}).\\ 
 \indent Recently, the attention has turned towards non-linear operators, especially the so-called $p$-Laplacian whose applications range from the description of non-Newtonian fluids (see \cite{newtonian}) to non-linear elasticity problems (\cite{elasticity}). Remarkably, the principal eigenvalue of the $p$-Laplacian could also be linked to the Ricci flow (see \cite{wu}). In 2003, Kawai and Nakauchi showed the lower bound $\lambda\geq \frac{1}{p-1}\frac{\pi_p^p}{(4D)^p}$ for the principal eigenvalue of the $p-$Laplacian with Neumann boundary conditions assuming non-negative Ricci-curvature and $p>2$. Here, $\pi_p$ denotes one half of the period of the so-called $p$-sine function. Similarly to \cite{Bakry1}, they used a gradient comparison technique which they proved with a maximum principle involving the $p$-Laplacian. In 2007, Zhang improved this result to $\lambda\geq (p-1)\pi_p^p/(2D)^p$ for any $p>1$, assuming non-negative Ricci-curvature and at least one point with positive Ricci-curvature. In 2012, assuming non-negative Ricci-curvature, Valtorta finally obtained the sharp estimate $\lambda \geq (p-1)\pi^p_p/D^p$ valid for the first eigenvalue of the $p-$Laplacian for any $p\in(1,\infty)$. The first improvement in his proof was a generalization of the celebrated Bochner formula to the linearization of the $p$-Laplacian, which yielded an improved gradient comparison with a one-dimensional model space using a maximum principle argument. The second improvement was a better estimate for the maximum of the eigenfunction, motivated by the techniques in \cite{Bakry1}. In \cite{val2}, these results were extended to negative lower bounds for the Ricci-curvature using similar techniques with a slightly more complicated model space. \\ 
 \indent In this paper, we combine the approaches of \cite{Bakry1} and \cite{val1} and define a non-linear $p$-operator $L_p$, which arises from a generic elliptic diffusion operator $L$ defined on a compact differentiable manifold which is allowed to have a convex boundary (in a suitable sense). More precisely, we define $L_pu:=\Gamma(u)^\frac{p-2}{2}(Lu+(p-2)H_u(u,u)/\Gamma(u))$, where $\Gamma$ is the so-called Carr\'e du Champ operator, which can be seen as a metric on $T^*M$ induced by $L$. Using intrinsic objects similar to \cite{Bakry1} and constraints which solely depend on the operator $L$, we generalize the approach by \cite{val1}. In particular, we prove the following theorem: 
 \begin{theorem}
 	Let $M$ be a a compact and connected smooth manifold with an elliptic diffusion operator $L$ with a smooth and L-invariant measure $m$ and let $\lambda$ be the principal Neumann-eigenvalue   of the $p-$operator $L_pu:=\Gamma(u)^\frac{p-2}{2}(Lu+(p-2)H_u(u,u)/\Gamma(u))$. If $L$ satisfies $\operatorname{BE}(0,N)$ for some $N\in[1,\infty)$ and if the boundary of $M$ is  either empty or  convex,  then the sharp estimate
 	\begin{align*}
 	\lambda\geq (p-1)\frac{\pi_p^p}{D^p}
 	\end{align*}
 	holds, where $D$ is the diameter associated with the intrinsic metric $d$. Equality holds if and only if $L$ satisfies BE$(0,1)$.
 	\label{maintheorem1}
 \end{theorem}
 \indent We emphasize that the constraints are satisfied by a much larger class than the Laplace-Beltrami operators. For instance, every operator, which satisfies the condition BE$(\kappa,N_0)$ for some $\kappa>0$ and $N_0<\infty$, satisfies BE$(0,N)$ for some large $N>N_0$ if $M$ is compact. In particular, our result applies to certain Bakry-Emery Laplacians in the form of $L=\Delta_g+\nabla\phi\cdot\nabla $.\\
 The rest of this paper is organized as follows:
 in section 2, we briefly review the definitions of $\Gamma_2$-calculus introduced by Bakry, Emery, and Ledoux (\cite{emery}, \cite{bakry2}), define the $p$-operator and discuss the existence of the first eigenvalue. In section 3, we prove Theorem 1.1. More precisely, we prove a generalized $p$-Bochner formula for the linearization $\mathcal{L}^u_p$ of $L_p$ and use a self-improvement property of the $\operatorname{BE}(0,N)$ condition similar to \cite{Bakry1} to obtain a good estimate for $\mathcal{L}^u_p(\Gamma(u)^\frac{p}{2})$. Next, we use a maximum principle argument in the fashion of \cite{val1} to compare the gradient of the eigenfunction with a suitable one-dimensional model space. This also yields a maximum comparison similar to \cite{val1} and allows us to obtain a sharp estimate for the principal eigenvalue. In the case $N=\infty$ the maximum comparison breaks down and we obtain a weaker estimate, which we expect not to be sharp.  We will also address the question of equality: While the estimate is sharp regardless of the dimension, equality can only be attained if $\dim(M)=1$ and $L=\Delta_g$ for some Riemannian metric $g$. In particular, $L$ must satisfy the condition BE$(0,1)$, which is in line with the results obtained by Hang and Wang in \cite{hang} and Valtorta in \cite{val1}.  Similarly to the $p$-Laplacian, we expect the introduced $p$-operator to be useful to model various non-linear problems in physics. \\
 \indent In section 4, we turn our attention towards non-symmetric operators of the form $L=\Delta_g+X\cdot\nabla$ for some Riemannian metric $g$, that is, operators that might not possess an invariant measure. Non-symmetric operators are important in quantum mechanics (see \cite{nonqm}) and can be used to describe damped oscillators.  It is a well known fact that such operators have a discrete and typically complex spectrum.
 Only recently, in \cite{andrews}, Andrews and Ni proved a lower bound for the first eigenvalue of the so-called Bakry-Emery Laplacian, which is symmetric with respect to a conformal measure. The main ingredient in their proof is a comparison theorem for the modulus of continuity of solutions of the heat equation with drift, which is a variation of the argument in the celebrated proof of the fundamental gap conjecture by Andrews and Clutterbuck (see \cite{andrews2}). In 2015, Wolfson proved a generalized fundamental gap conjecture for non-symmetric Schr\"odinger operators. In the same spirit, we generalize the results obtained in \cite{andrews} to non-symmetric operators; that is, we do not require the first order part to be the gradient of a function. Here, we restrict ourselves to the linear case $p=2$. More precisely, we will prove the following theorem:
 \begin{theorem}
 	Let $(M,g)$ be a compact and connected Riemannian manifold, possibly with a strictly convex boundary together with a non-symmetric diffusion operator $L=\Delta_g+X\cdot\nabla$ satisfying $\operatorname{BE}(a,\infty)$. Let $\lambda$ be a non-zero eigenvalue of $L$ with Neumann Boundary conditions. Then one has the  estimate
 	$$
 	\operatorname{Re}(\lambda)\geq \frac{\pi^2}{D^2}+\frac{a}{2},
 	$$
 	where $D$ is the diameter of the Riemannian manifold.
 	\label{maintheorem2}
 \end{theorem}
 The main difference to the result in \cite{andrews} is that we do not impose any additional constraints on the first order term $X$, whereas \cite{andrews} requires $X$ to be the gradient of a function $\phi$. Such operators are called Bakry-Emery Laplacians and have the invariant measure $e^{-\phi}dVol$, which implies that the spectrum is real. If $X$ is not a gradient, then even the principal eigenvalue is generally complex. \\ \indent We follow \cite{andrews} to prove the theorem: indeed, we compare the decay of a heat flow associated with the operator $L$ to the decay of the heat flow in a one-dimensional model space, and the estimate is obtained by using a maximum principle argument. Contrary to the proof of the first theorem, the argument relies heavily on the Riemannian geometry induced by the operator $L$. The distortion of the geometry induced by the first order term $X$ will only play a minor part. Although the eigenvalue comparison with the model space is sharp, we prefer to state the lower bound $\pi^2/D^2+a/2$, which is not sharp, but more useful. A better lower bound can be obtained through a better understanding of the model function.  \\
 \indent For the non-linear case $p\neq 2$, it is not even clear if an eigenvalue exists. Even if the existence could be established, our methods would not be applicable: 
 the approach we used to obtain the sharp estimate for the principal eigenvalue of the non-linear $p$-operator explicitly exploits that $L$ is self-adjoint with respect to the invariant measure $m$, particularly that $\lambda$ is real, whereas the argument in our second approach relies heavily on the linearity of $L$. Hence,  it does not seem that they could be generalized to the non-symmetric, non-linear case.

\section{Preliminaries}
\label{preliminaries}
\subsection{The geometry of diffusion operators}

We repeat the basic definitions of $\Gamma_2$-calculus in the setting of a smooth manifold and refer to \cite{bakry2} for a good introduction. Let $M$ denote a connected smooth manifold, which is allowed to have a boundary. 
\begin{definition}A linear second-order operator $L:C^\infty(M)\to C^\infty(M)$ is defined to be an elliptic diffusion operator \label{diffusion.operatprs}
	if for any smooth function $\psi:\mathbb{R}^r\to\mathbb{R}$, any $f,f_i\in C^{\infty}(M)$
	and at every point $x\in M$ one has
	$$
	L(\psi(f_1,\dots,f_r))=\sum_{i=1}^{r}\partial_i\psi L(f_i)+\sum_{i,j=1}^{r}\partial_i\partial_j\psi\Gamma(f_i,f_j)
	$$
	and $\Gamma(f,f)\geq0$ with equality if and only if $df=0$. Here,
	$\Gamma$ denotes the so-called Carr\'e du Champ operator; that is,
	$$
	\Gamma(f,g):=\frac12\bigg(L(fg)-fLg-gLf\bigg).
	$$
	
\end{definition}
One easily verifies that diffusion operators are exactly all linear, possibly degenerate elliptic differential operators depending only on the first and second derivatives but not on the function itself. If $L$ is elliptic, then the operator $\Gamma$ induces a Riemannian metric $g$ which satisfies $\Gamma(f,f)=|\nabla _g f|^2$ and $L=\Delta_g+X\cdot\nabla_g$ for some vector field $X$. If $C=\nabla_g \phi$ for some smooth function $\phi$ then $L$ is called Bakry-Emery Laplacian and one easily verifies that $L$ is the Laplace-Beltrami operator of some metric $\tilde g$ if and only if $\phi$ is constant. 
\par In order to view the pair $(M,L)$ as a metric measure space we need the following definitions: 
\begin{definition}
	\label{invariance}
	We say that a locally finite Borel measure $m$ is $L$-invariant if there is a generalized function $\nu$ such that
	$$
	\int_M \Gamma(f,h)dm=-\int_M fLhdm+\int_{\partial M}g\Gamma(h,\nu)dm
	$$
	holds for all smooth $f,g$. $\nu$ is called the outward normal function and is defined to be a set of pairs $(\nu_i,U_i)_{i\in I}$ for a covering $U_i$ of $\partial M$ such that $\nu_i\in C^\infty(U_i)$ and $\Gamma(\nu_i-\nu_j,\cdot)|_{U_i\cap U_j}\equiv 0$. Furthermore, we say that $m$ is smooth if the pushforward of $m$ by any chart of $M$ has a smooth density with respect to the Lebesque measure. 
\end{definition}

\begin{definition}
	The intrinsic distance $d:M\times M \to [0,\infty]$ is defined by
	$$
	d(x,y):=\operatorname{sup}\bigg\{f(x)-f(y) \big| f\in C^\infty(M), \Gamma(f)\leq1\bigg\}
	$$
	and the diameter of $M$ by $D:=\sup\{d(x,y)|x,y\in M\}$.
\end{definition}
The above definitions can be iterated to produce the Hessian and $\Gamma_2$-operator. These operators will then induce a geometry on $(M,L)$.
\begin{definition}
	For any $f,a,b\in C^{\infty}(M)$, we define the Hessian by
	$$
	H_f(a,b)=\frac12\bigg(\Gamma(a,\Gamma(f,b))+\Gamma(b,\Gamma(f,a))-\Gamma(f,\Gamma(a,b))\bigg)
	$$
	and the $\Gamma_2-$operator by
	$$
	\Gamma_2(a,b)=\frac12\bigg(L(\Gamma(a,b))-\Gamma(a,Lb)-\Gamma(b,La)\bigg).
	$$
\end{definition}
The Hessian only depends on the second order terms of $L$ and thus can be seen as the Hessian of some Riemannian metric $g$. On the other hand, $\Gamma_2$ also depends on first order terms of $L$ and thus induces a geometry which in general cannot be seen as a Riemannian object.
\begin{definition}
	For $N\in[1,\infty]$, $x\in M$ and $f\in C^{\infty}(M)$, we define the $N-$Ricci-curvature pointwise by
	$$
	R_N(f,f)(x)=\inf\bigg\{\Gamma_2(\phi,\phi)(x)-\frac{1}{N}(L\phi)^2(x)\big|\phi\in C^{\infty}(M), \Gamma(\phi-f)(x)=0\bigg\}
	$$
	and the Ricci-curvature by $R:=R_\infty$, where we use the convention $1/\infty=0$.\\
	Let $k\in\mathbb{R}$ and $N\in[1,\infty]$. We say that $L$ satisfies 
	$\operatorname{BE}(k,N)$ (the \textit{Bakry-Emery curvature-dimension condition}) if and only if
	\begin{align}
	R_N(f,f)\geq k\Gamma(f). \label{BE}
	\end{align}
	for any $f\in C^{\infty}(M)$.
	This inequality is called the \textit{curvature dimension inequality}.
\end{definition}
\begin{remark}
	The Bochner-formula implies that $\Delta_g$ satisfies $\operatorname{BE}(\kappa,N)$ if and only if $\operatorname{ric}\geq \kappa$ and $\dim(M)\leq N$, but in general, the curvature-dimension condition does not have such an intuitive meaning. For instance, one can easily show that the Ornstein-Uhlenbeck operator $L(f):=\Delta f-x\cdot\nabla f$  on $M=\mathbb{R}^N$ induces the Euclidean distance and satisfies BE$(1,\infty)$. However, one has $R_N\equiv-\infty$ for all $N\in[1,\infty)$. So we see that the geometry induced by $L$ is very different from the geometric situation in $\mathbb{R}^N$ equipped with the standard inner product.
\end{remark}
\par \indent
In spite of such unexpected behaviour, the Ricci curvature turns out to be computable in an easier way. Indeed, in \cite{sturm1}, Sturm showed a generalized Bochner formula:
\begin{theorem}\label{2.bochner.formula}
	For any diffusion operator defined on a Riemannian manifold, we have
	$$
	\Gamma_2(f,f)=R(f,f)+|H_f|^2_{HS}
	$$
	for each $f\in C^\infty(M)$, where $|H_f|^2_{HS}$ denotes the square of the Hilbert-Schmidt norm of the Hessian, that is, $|H_f|^2_{HS}(x):=\operatorname{sup}\big\{\sum_{i,j=1}^{\tilde N(x)}|H_f(e_i,e_j)|^2(x)\big|\{e_i\} \text{ is an ONB of } \Gamma \text{ at } x\}$.
\end{theorem}
When studying the principal eigenvalue, the geometry of $\partial M$ will also play an important part, hence we make the following definition.
\begin{definition} \label{s.fund.form}
	Let $\nu$ be the outward normal function, $U\subset M$ be an open set,
	and $\phi,\eta\in C^\infty(U)$,  such that $\Gamma({\nu},\eta),\Gamma({\nu},\phi)\equiv0$ on $U\cap\partial M$. Then we define the second fundamental form on $\partial M$ by
	$$
	II(\phi,\eta)=-H_\phi(\eta,\nu)=-\frac12\Gamma(\nu,\Gamma(\eta,\phi)).
	$$
	If for any $\phi\in C^\infty(U)$ as above with $\Gamma(\phi)>0$ on $\partial M\cap U$ we have that
	$$
	II(\phi,\phi)\leq0
	$$
	on $\partial M \cap U$, then we say that $\partial M$ is convex or that $M$ has a convex boundary. If the inequality is strict, we say that $M$ has a strictly convex boundary.
\end{definition}
\par \indent
As an example, we consider the Bakry-Emery Laplacian $L=\Delta_{\bar g}+\nabla_{\bar g} \phi\cdot\nabla_{\bar g}$ for some Riemannian metric $\bar g$ and a smooth function $\phi$. We define the conformal metric $g:=e^{-\frac{2}{N}\phi}\tilde g$ and let $m=dVol_g,\bar m=dVol_{\bar g}$. Obviously, we have $e^\phi m=\bar m$. Using the divergence theorem one can easily see that $m$ is $L$-invariant. Furthermore, one can show that the Ricci-curvature of $L$ is the Ricci-curvature of $\bar g$ plus a first-order perturbation of the metric $\bar g$ which can be controlled by the $C^2$-norm of $\phi$. So if  $\operatorname{ric}_{\bar g}\geq\kappa \bar g$ and $\phi$ is not too big with respect to its $C^2$ norm, we obtain  $R\geq \kappa' \Gamma$ for some $0<\kappa'<\kappa$.  Now the Bochner formula, the inequality $\Delta_{\tilde g}=(\operatorname{tr}H)^2\leq \dim(M)|H|^2_{HS}$, and the Cauchy-Schwarz inequality imply that $L$ satisfies BE$(\kappa ',N)$ for some finite $N>\dim(M)$.

\subsection{$p$-operators}
In this section, we will assume that $M$ is a closed manifold and $L$ an elliptic diffusion operator with a smooth $L$-invariant measure $m$. The measure $m$ induces the space $W^{1,p}(M)$ and we can use the Riemannian metric induced by $L$ to define the spaces $C^{k,\alpha}(M)$.  We define the $p-$Operator to be the natural generalization of the $p$-Laplacian.
\begin{definition}
	Let $p\in(1,\infty)$ and $f\in C^\infty(M)$. We define the $p-$operator $L_p$ by
	$$
	L_pf(x):=\begin{cases}
	&\Gamma(f)^{\frac{p-2}{2}}\bigg(Lf+(p-2)\frac{H_f(f,f)}{\Gamma(f)}\bigg)\bigg|_x \text{ if } \Gamma(f)(x)\neq 0 \\&0 \qquad\qquad\qquad\qquad\qquad\qquad\qquad\text{ else}
	\end{cases}
	$$
	and the main part of its linearization at $f$ by
	$$
	\mathcal{L}^f_p(\psi)=\begin{cases}
	&\Gamma(f)^{\frac{p-2}{2}}\bigg(L\psi+(p-2)\frac{H_\psi(f,f)}{\Gamma(f)}\bigg)\bigg|_x \text{ if } \Gamma(f)(x)\neq 0 \\&0 \qquad\qquad\qquad\qquad\qquad\qquad\qquad\text{ else}
	\end{cases}
	$$
	for any $\phi\in C^{\infty}(M)$.
\end{definition}
\begin{remark} Since $L_p$ is quasi-linear, we have $\mathcal{L}^f_p(f)=L_p(f)$.
	The $p-$operator can often be seen as a first order perturbation of a $p-$Laplacian: Let $L=\Delta_{\tilde g}+\nabla\phi\cdot\nabla$ for some Riemannian metric $\tilde g$. Then we have $\Gamma (f)=|\nabla f|_{\tilde g}^2$ and $H=H^{\tilde g}$. Hence, we have
	$$
	L_p(f)=\Delta^p_{\tilde g}f+|\nabla f|^{p-2}_{\tilde g}\nabla \phi\cdot\nabla f.
	$$
\end{remark}
Next we would like to define an eigenvalue of $L$. Adjusting for scaling factors, we expect an eigenfunction $u$ with eigenvalue $\lambda$ to satisfy
\begin{align*}
\begin{cases}
&L_pu=-\lambda u|u|^{p-2} \text{ on } M^\circ\\
&\Gamma(u,\tilde \nu)=0 \text{ on }\partial M
\end{cases}
\end{align*}
in a suitable sense. Since $L_pu$ will not be defined everywhere, we have to integrate the equation. We start with the following lemma which can easily be deduced from the invariance of $m$. 
\begin{lemma}\label{int.by.parts}
	Let $\phi\in{C^\infty(M)}$ and $f\in C^2(\operatorname{supp}(\phi))$ as well as $\Gamma(f)>0$ on $\operatorname{supp}(\phi)$. Then we have
	$$
	\int_M \phi L_pudm=-\int_M\Gamma(f)^{\frac{p-2}{2}}\Gamma(f,\phi)dm+\int_{\partial M} \Gamma(f,\tilde{\nu})\Gamma(f)^{\frac{p-2}{2}}\phi dm'.
	$$
\end{lemma} 
This formula suggests the following weak eigenvalue equation. Homogeneous Neumann boundary conditions will arise naturally from this definition if $M$ has a non-empty boundary.  
\begin{definition}
	We say that $\lambda$ is an eigenvalue of $L_p$ if there is a
	$u\in {W^{1,p}(M)}$, such that for any $\phi\in {C^{\infty}(M)}$ the following identity holds:
	$$
	\int_M\Gamma(u)^{\frac{p-2}{2}}\Gamma(u,\phi)dm=\lambda\int_M \phi u|u|^{p-2}dm.
	$$
	By density, this also holds for all $\phi\in W^{1,p}(M)$.
\end{definition}
Choosing $\phi\equiv 1$, we see that $\int_M u|u|^{p-2}dm=0$ unless $\lambda=0$. This a priori constraint allows us to show the existence of the principal eigenvalue using standard variational techniques.
\begin{lemma} \label{existence.regularity}
	Let $M$ be a compact and smooth Riemannian manifold with an elliptic diffusion operator $L$ and a smooth $L$-invariant measure $m$. Then the principal eigenvalue of $L_p$ (with Neumann boundary conditions) is well-defined and the eigenfunction $u$ is in
	$C^{1,\alpha}(M)$ for some $\alpha>0$. $u$ is smooth near points $x\in M$ satisfying
	$\Gamma(u)(x),u(x)\neq 0$, and in $C^{3,\alpha}$ and $C^{2,\alpha}$
	for $p>2$ and $p<2$, respectively, near points with $\Gamma(u)(x)\neq0$ and $u(x)= 0$.
\end{lemma}
\begin{proof}
	This follows from using the direct method of the calculus of variations to compute the quantity
	$$
	\sigma_p:=\inf\bigg\{\frac{\int_M \Gamma(u)^{\frac{p}{2}}}{\int_M |u|^{p}} \bigg| u\in W^{1,p}(M),   \int_M u|u|^{p-2}=0\bigg\}.
	$$
	The regularity statement follows from \cite[Theoem 1]{tol} and standard Schauder estimates.
	
\end{proof}
\begin{remark}
	\label{strong.ev.eqn}
	Near interior points where $\Gamma(u)$ does not vanish, we have $L_p(u)=-\lambda u|u|^{p-1}$ and one easily shows that $\Gamma(u,\nu)_{|\partial M}\equiv 0$. Since
	$\int_M u|u|^{p-2}dm=0$, $u$ must change its sign and the eigenvalue equation is invariant under rescaling, so we can
	assume without loss of generality that $\min u=-1$ and $\max u\leq 1$.
\end{remark}

\section{Eigenvalue Estimate for the $p$-Operator}
\label{eigenvalue.estimate}
In this section, we prove a sharp estimate for the principal eigenvalue of $L_p$. Throughout this section, we will assume that $M$ is compact and connected and $L$ an elliptic diffusion operator with invariant measure $m=dVol$. If $M$ has a non-empty boundary, we assume it to be  convex. We define $\lambda$ to be the principal eigenvalue of $L_p$ with Neumann boundary conditions and $u$ the corresponding eigenfunction, with $\min u=1$ and $\max u\leq 1$. Finally, we assume that $L$ satisfies $\operatorname{BE}(0,N)$ for some $N\in[1,\infty]$. If $M=[-D/2,D/2]$, $L=\Delta$ and $m=\mathcal{L}^1$ then one can easily show that  $\lambda=(p-1)\frac{D^p}{\pi_p^p}$ and $u(\cdot)=\sin_p(\frac{\pi_p}{D}\cdot)$, where
$$
\frac12\pi_p =\int_{0}^{1}\frac{1}{(1-s^p)^{\frac{1}{p}}}
$$
and the $p$-sine function is defined implicitly for $t\in[-\frac{\pi_p}{2},\frac{\pi_p}{2}]$ by
$$
t=\int_{0}^{\sin_p(t)}\frac{1}{(1-s^p)^{\frac{1}{p}}}ds.
$$
It is natural to think that this operator minimizes the principal eigenvalue amongst all admissible operators. However, they do not always turn out to be suitable comparison models. Given a model function $w$, we would like to use the function $u\circ w^{-1}$ to estimate the diameter from below. This will only be optimal if $\max u=\max w$. If $\max u=1$, the one-dimensional eigenvalue equation will be a good comparison model, otherwise, we will consider a relaxed equation dampening the growth of $w$ to give $\max w=\max u$.
More precisely, we follow \cite[section 5]{val1}
and consider for any $n\in(1,\infty)$ the equation
\begin{align}
\label{model.ode}
\begin{cases}
&\frac{\partial}{\partial t} \big(w'|w'|^{p-2}\big)-Tw'|w'|^{p-2}+\lambda w|w|^{p-2}=0 \text{ on } (0,\infty)\\&w(a)=-1 \qquad w'(a)=0
\end{cases}
\end{align} 
where $a\geq 0$ and $T$ is a solution of the differential equation $T^2/(n-1)=T'$, that is, either $T=-(n-1)/t$ or $T\equiv 0$. If $T\equiv 0$, this is simply the eigenvalue equation, otherwise, it can be seen as a relaxed eigenvalue equation. For any $n\in(1,\infty)$, we will denote the solution corresponding to $T=-(n-1)/t$ and $a\geq 0$ by $w_a$ and for ease of notation we define $w_\infty$ to be the solution corresponding to $T\equiv 0$. \\\indent
We define $\alpha:=(\lambda/(p-1))^{\frac{1}{p}}$ and note that
the differential equation implies that for any $t>a$ which is close to $a$, one has $w_a''(t)>0$, so there exists a first time $b=b(a)$ such that $w_a'(b(a))=0$ and $w_a'>0$ in $(a,b(a))$. In particular, the restriction
of $w_a$ to $[a,b(a)]$ is invertible and we identify $w_a=w_a|_{[a,b(a)]}$. Finally, we define $\delta(a):=b(a)-a$ to be the length of the interval and  $m(a)=w_a(b(a))$ to be the maximum of $w_a$. As we have seen, $\delta(\infty)=\pi_p/\alpha$ and $m(\infty)=1$ for any $\lambda>0$.
\\ \indent In order to compare the maximum and gradient of an eigenfunction $u$ and $w_a$, we need to understand the behaviour of the solutions of Equation (\ref{model.ode}) and the asymptotic behaviour of the functions $m$ and $\delta$. This is done in the following two theorems.
\begin{theorem} 
	For any $\lambda>0$, equation (\ref{model.ode}) has a solution $w_a\in C^1(0,\infty)$ which is also in $C^0([0,\infty)])$ if $a=0$. Furthermore, $w_a'|w_a'|^{p-2}\in C^1((0,\infty))$. The solution depends continuously on
	$n$, $\lambda$, and $a$ in terms of local uniformal convergence of $w_a$ and $w_a'$ in $(0,\infty)$. Additionally, for each solution there is a sequence $t_i\in(0,\infty)$ with $t_i\to\infty$ and $w_a(t_i)=0$.
\end{theorem}
\begin{proof}
	This follows from basic ODE theory (see \cite[section 5]{val1}).  \end{proof}
\begin{theorem} \label{ode.asymptotics}
	For any $n>1$, the function $\delta(a)$ is continuous on $[0,\infty)$
	and  strictly greater than $\pi_p/\alpha$. Furthermore, for any $a\in[0,\infty)$ we have
	$$
	\lim_{a\to\infty}\delta(a)=\frac{\pi_p}{\alpha},\qquad m(a)<1, \qquad
	\lim_{a\to\infty}m(a)=1.
	$$
\end{theorem}
\begin{proof}
	Continuity just follows from continuous dependence of the data. The rest of the proof is a bit technical, but straightforward: The statement is true for $a=\infty$, hence it suffices to show that $\delta$ is decreasing and $m$ is increasing (see \cite[section 5]{val1}). 
\end{proof} 
\subsection{Gradient comparison}
In this subsection, we compare the gradient of the eigenfunction with the gradient of the model function in the one-dimensional model space. Following \cite{val1}, we will prove the gradient comparison using a maximum principle involving the linearization $\mathcal{L}^u_p$. The first step  
to generalize the $p$-Bochner formula:
\begin{lemma}[$p$-Bochner formula] \label{p.bochner.formula} Let $u\in C^{1,\alpha}(M)$ be the first eigenfunction of $L_p$. Let $x\in M$ be a point such that 
	$\Gamma(u)(x)\neq 0$ and $u(x)\neq 0 $ if $1<p<2$. Then we have the p-Bochner formula
	\begin{align*}
	\frac{1}{p}\mathcal{L}^u_p(\Gamma(u)^{\frac{p}{2}})=&\Gamma(u)^{\frac{p-2}{2}}\bigg(\Gamma(L_pu,u)-(p-2)L_puA_u\bigg)\\&+\Gamma(u)^{p-2}\bigg(\Gamma_2(u)+p(p-2)A_u^2\bigg),
	\end{align*}
	where $A_u=H_u(u,u)/\Gamma(u)$.
\end{lemma}
\begin{proof}
	We can assume that $\Gamma(u)(x)=1$, since both sides scale in the same way. In an environment of $x$ we have that $u\in C^{3,\alpha}$, so we can perform all of the following computations. Since $L$ is a diffusion operator, we have
	\begin{align*}
	L(\Gamma(u)^{\frac{p}{2}})&=\frac{p}{2}L(\Gamma(u))+\frac{1}{4}p(p-2)\Gamma(\Gamma(u),\Gamma(u))\\&=p(\Gamma_2(u)+\Gamma(u,Lu))+\frac{1}{4}p(p-2)\Gamma(\Gamma(u),\Gamma(u)).
	\end{align*}
	For the next calculation we use the chain rule (see {Lemma \ref{chain.rule}} below)
	\begin{align*}
	H_{\Gamma(u)^{\frac{p}{2}}}(u,u)&=\Gamma(u,\Gamma(u,{\Gamma(u)^{\frac{p}{2}}}))
	-\frac{1}{2}\Gamma(\Gamma(u)^{\frac{p}{2}},\Gamma(u))
	\\&=\frac{p}{2}\Gamma(u,\Gamma(u,\Gamma(u))\Gamma(u)^{\frac{p-2}{2}})
	-\frac{1}{4}p\Gamma(\Gamma(u),\Gamma(u))
	\\
	&=\frac{p}{2}\Gamma(u,\Gamma(u,\Gamma(u)))+p(p-2)H_u(u,u)^2
	-\frac{1}{4}p\Gamma(\Gamma(u),\Gamma(u)).
	\end{align*}
	Finally,
	\begin{align*}
	\Gamma(L_pu,u)=&\Gamma(\Gamma(u)^{\frac{p-2}{2}}L(u)+(p-2)\Gamma(u)^{\frac{p-4}{2}}H_u(u,u),u)
	\\
	=&\Gamma(u,Lu)+\frac{1}{2}(p-2)Lu\Gamma(u,\Gamma(u))+(p-2)\frac12\Gamma(u,\Gamma(u,\Gamma(u)))\\&+(p-2)\frac{p-4}{2}H_u(u,u)\Gamma(u,\Gamma(u))
	\\=&\Gamma(u,Lu)+(p-2)LuH_u(u,u)+(p-2)\frac12\Gamma(u,\Gamma(u,\Gamma(u)))\\&+
	(p-2)(p-4)H_u(u,u)^2.
	\end{align*}
	Now, using that $A_u=H_u(u,u)$ for $\Gamma(u)=1$
	and the fact that $(p-2)(p-4)-(p-2)^2+p(p-2)=(p-2)^2$ for the $A_u$ terms on
	the right-hand side, the claim follows.  \end{proof} 
In the proof, we have used
\begin{lemma}\label{chain.rule}
	Let $a,b,c\in C^2(M)$ and $f:\mathbb{R}\to\mathbb{R}$ be smooth, then
	$$\Gamma(a,bc)=\Gamma(a,b)c+\Gamma(a,c)b
	$$
	and $$
	\Gamma(f(a),b)=f'(a)\Gamma(a,b).
	$$
	Furthermore, we have
	$$
	H_{ab}(u,u)=bH_a(u,u)+aH_b(u,u)+2\Gamma(u,a)\Gamma(u,b)
	$$
	and 
	$$
	H_{f(a)}(u,u)=f'(a)H_a(u,u)+f''(a)\Gamma(u)^2.
	$$
\end{lemma}
\begin{proof}
	This follows directly from the diffusion property of $L$. \end{proof} 
\indent
To apply the maximum principle argument, we will have to estimate the term $\mathcal{L}^u_{p}(\Gamma(u)^{\frac{p}{2}})$ from below. Afterwards, we will rewrite the inequality in terms of a model function which satisfies a certain differential equation and can be expressed in terms of $u$. 
We can replace the $L_p(u)$ terms by $-u|u|^{p-1}$ and bearing in mind that the first derivatives of a function vanish at a maximum point, we will be able to replace the $A_u$ terms, too. Hence, the last ingredient is a good estimate for the $\Gamma_2$ term.
\begin{lemma}\label{improved.BE}
	Let $u$ be the first eigenfunction of $L_p$ and consider a point $x\in M$ with  $\Gamma(u)(x)\neq0$. Assume that $L$ satisfies $BE(0,N)$ for some $N\in[1,\infty]$. Let $n\geq N$. If $n>1$, we have
	$$
	\Gamma(u)^{p-2}\bigg(\Gamma_2(u,u)+p(p-2)A_u^2\bigg)\geq\frac{(L_p(u))^2}{n}+\frac{n}{n-1}\bigg(\frac{L_p(u)}{n}-(p-1)\Gamma(u)^{\frac{p-2}{2}}A_u\bigg)^2,
	$$	
	where we use the convention $1/\infty=0$ and $\infty/\infty=1$.
	For $n=1$, we get
	$$
	\Gamma(u)^{p-2}\bigg(\Gamma_2(u,u)+p(p-2)A_u^2\bigg)\geq{(L_p(u))^2}.
	$$	
\end{lemma}
\begin{proof} Since $\Gamma(u)(x)\neq 0$, it holds that $u\in C^{2,\alpha}$ near $x$, so all of the following calculations can be performed. Since both sides scale in the same way, we can assume that
	$\Gamma(u)(x)=1$. The condition BE$(0,N)$ implies BE$(0,n)$, so we can assume that $n=N$. If $N=1$, then the condition $\operatorname{BE}(0,N)$ implies $\operatorname{BE}(0,\dim(M))$ which gives $A_u=\operatorname{tr}H_u=Lu$ and the result follows immediately by using the estimate $\Gamma_2(u)\geq (Lu)^2$. If $N=\infty$, we use the trivial estimate $\Gamma_2(u)\geq|H_u|^2_{HS}\geq A_u^2$, so we can assume that $1<N<\infty$. The idea is that the curvature-dimension inequality has a self-improvement property (see \cite{Bakry1} for the linear case): Let $v$ be an arbitrary smooth function. Since $L$ satisfies $\operatorname{BE}(0,N)$, we have at $x$
	\begin{align*}
	\Gamma_2(v,v)\geq \frac{1}{N}(Lv)^2&=\frac{1}{N}(\mathcal{L}^u_p(v))^2-\frac{2(p-2)}{N}LvH_v(u,u)-\frac{(p-2)^2}{N}H_v(u,u)^2 \\
	\\&=\frac{1}{N}(\mathcal{L}^u_p(v))^2-\frac{2(p-2)}{N}\mathcal{L}^u_pvH_v(u,u)+\frac{(p-2)^2}{N}H_v(u,u)^2
	\\&=:\frac{1}{N}(\mathcal{L}^u_p(v))^2+C(v,v).
	\end{align*}
	Now we define the quadratic form $B(v,v)=\Gamma_2(v,v)-(\mathcal{L}^u_p(v))^2/N-C(v,v)$ and let $\phi\in C^\infty(\mathbb{R})$ be a smooth function. By assumption, we have $B(\phi(u),\phi(u))(x)\geq 0$. Using that $\Gamma(u)(x)=1$ and $H_u(u,u)(x)=A_u(x)$, we have the following identities at $x$: $$\Gamma_2(\phi(v),\phi(v))=\phi'^2\Gamma_2(u,u)+2\phi'\phi''A_u+\phi'',$$
	$$
	\mathcal{L}^u_p(\phi(u))=\phi'L_p(u)+(p-1)\phi'',\qquad H_{\phi(u)}(u,u)=\phi'A_u+\phi''.
	$$
	This gives
	\begin{align}
	B(\phi(u),\phi(u))=&\phi'^2 \notag B(u,u)+2\phi'\phi''\bigg(A_u-\frac{p-1}{N}L_pu+\frac{(p-2)}{N}\big((p-1)A_u+L_pu\big)\\&-\frac{(p-2)^2}{N}A_u\bigg)+\phi''^2\bigg(1-\frac{1}{N}(p-1)^2+2\frac{(p-2)(p-1)}{N}-\frac{(p-2)^2}{N}\bigg) \notag
	\\=&\phi'^2 B(u,u)+2\phi'\phi''\bigg(-\frac{L_p u}{N}+\big(1+\frac{p-2}{N}\big)A_u\bigg)+\phi''^2\frac{N-1}{N}. \label{discriminant}
	\end{align}
	Now for any $a$, we can choose a function $\phi$ such that $\phi'(u(x))=a$ and $\phi''(u(x))=1$, so equation (\ref{discriminant}) becomes a non-negative, quadratic polynomial in $a$ and hence must have a non-negative discriminant, that is:
	$$
	4B(u,u)\frac{N-1}{N}\geq 4 \bigg(\frac{L_p u}{N}-\big(1+\frac{p-2}{N}\big)A_u\bigg)^2.
	$$
	This, however, is equivalent to
	\begin{align*}
	\Gamma_2(u,u)+p(p-2)A_u^2\geq& \frac{1}{N}(L_pu)^2+p(p-2)A_u^2+C(u,u)\\&+\frac{N}{N-1}
	\bigg(\frac{L_p u}{N}-\big(1+\frac{p-2}{N}\big)A_u\bigg)^2
	\\=&\frac{1}{N}(L_pu)^2+\frac{N}{N-1}
	\bigg(\frac{L_p u}{N}-(p-1)A_u\bigg)^2,
	\end{align*}
	where the last equality can be verified by direct computation.
\end{proof}
\begin{remark}
	The advantage of the self-improvement property is that it automatically gives a sharp estimate which is not immediate if we chose a local framework and discard certain terms. Although the estimate for $N=1$ seems slightly weaker, it is in fact just as strong since $L_pu-(p-1)\Gamma(u)^\frac{p-2}{2}A_u\equiv 0$ in one dimension.
\end{remark}
Before we prove the gradient comparison, we summarize the results we have
obtained so far:
\begin{corollary}\label{estimate.summary}
	Let $u$ be an eigenvalue of $L_p$, $x\in M$ and $\Gamma(u)(x)\neq 0$ and let $n\geq N$ and $n>1$. If $1<p<2$, let $u(x)\neq 0$. Then we have at $x$
	\begin{align*} 
	\frac{1}{p}\mathcal{L}^u_p(\Gamma(u)^{\frac{p}{2}})
	\geq &-\lambda(p-1)|u|^{p-2}\Gamma(u)^{\frac{p}{2}}+\lambda(p-2)u|u|^{p-2}
	\Gamma(u)^{\frac{p-2}{2}}A_u+\frac{\lambda^2u^{2p-2}}{n}\\&+\frac{\lambda^2u^{2p-2}}{n(n-1)}
	+2\frac{(p-1)\lambda}{n-1}u|u|^{p-2}\Gamma(u)^{\frac{p-2}{2}}A_u
	+\frac{n}{n-1}(p-1)^2\Gamma(u)^{\frac{2p-4}{2}}A_u^2
	\\=&-\lambda(p-1)|u|^{p-2}\Gamma(u)^{\frac{p}{2}}+\lambda \frac{(n+1)(p-1)-(n-1)}{(n-1)}u|u|^{p-2}
	\Gamma(u)^{\frac{p-2}{2}}A_u\notag\\&+\frac{\lambda^2u^{2p-2}}{n-1}
	+\frac{n}{n-1}(p-1)^2\Gamma(u)^{\frac{2p-4}{2}}A_u^2. \notag
	\end{align*}
\end{corollary}
\begin{proof}
	This follows directly from  
	{Lemma \ref{p.bochner.formula}}, {Lemma \ref{improved.BE}}, and the strong eigenvalue equation $L_p(u)=-\lambda u|u|^{p-2}$. We remark that $\Gamma(L_p,u)=\Gamma(-\lambda u|u|^{p-2},u)=-\lambda(p-1)|u|^{p-2}\Gamma(u)$. 
\end{proof} 
We are now able to prove the gradient comparison with a  suitable one-dimensional
model function. The proof is motivated by \cite[Theorem 4.1]{val1}.
\begin{theorem}[Gradient comparison theorem]\label{grad.comp.thm}
	Let $\lambda$ be the principal eigenvalue of $L_p$ with Neumann boundary conditions and $u$ be a corresponding eigenfunction. Assume that $L$ satisfies BE$(0,N)$.  Let $n\geq N$ with $n>1$, $a\geq 0$, and
	$w=w_a$ be the solution of the one-dimensional model equation \ref{model.ode},
	with either $T=-(n-1)/t$ or $T\equiv 0$. If $n=\infty$, we let $T\equiv 0$. Let $b=b(a)$ be the first
	root of $w'$ after $a$, as above. Now, if 
	$[\operatorname{min}(u),\operatorname{max}(u)]\subset[w(a),w(b)]$, then the following inequality holds on all of $M$:
	$$\Gamma (w^{-1} \circ u)\leq1.$$ 
\end{theorem}
\begin{proof} We can assume that $[\operatorname{min}(u),\operatorname{max}(u)]
	\subset(w(a),w(b))$ by replacing $u$ by $\xi u$ and  letting $\xi\nearrow 1$ afterwards. The regularity theory for ordinary differential equations gives that $w$ is smooth on $(a,b)$. Using the chain rule we see that it is equivalent to prove
	$$
	\Gamma(u)^{\frac12}(x)\leq w'(w^{-1}(u(x)))
	$$
	for all $x\in M$. In order to prove this,
	let $\phi(u(x)):=		w'(w^{-1}(u(x)))^p$ and
	$\psi\in C^2(\mathbb{R})$
	be a positive function which will be specified later. Define
	$$
	F:=\psi(u)(\Gamma (u)^{\frac{p}{2}}-\phi(u)).
	$$
	It suffices to show that $F\leq 0$. Since $M$ is compact, $F$ attains its maximum
	in $x\in M$. Furthermore, we have $\phi(u)>0$, so it suffices to consider the case
	$\Gamma (u)(x)>0$. If $p>2$, then we have $u\in C^{3,\alpha}$ around $x$ and all the following computations can be performed. We will explain below how to modify the proof in the case $1<p<2$. At the point $x$, we have 
	$$
	\Gamma(F,u)(x)=0, \qquad
	\mathcal{L}^u_p(F)(x)\leq 0.
	$$
	This is obvious if $x$ lies in the interior: $\Gamma$ is induced by a Riemannian metric on $T^*M$ so we get $\Gamma(F,\cdot)(x)=0$, which implies the first identity. On the other hand, $\mathcal{L}^u_p$ is elliptic away from critical points of $u$ and the first-order derivatives of $F$ vanish, which implies the inequality $\mathcal{L}^u_p(F)(x)\leq 0$. If $x$ lies on the boundary, we need to be more careful: it is immediate that $\Gamma(F,\cdot)(x)$ vanishes in all directions tangent to the boundary, in particular, $\Gamma(F,u)(x)=0$ since $\Gamma(u,\tilde \nu)|_{\partial M}\equiv 0$. Moreover, the Neumann boundary conditions and the convexity of $\partial M$ imply at $x$
	\begin{align}
	0\leq
	\Gamma(F,\tilde \nu)&=\psi'(u)\Gamma(u,\tilde \nu)\frac{F}{\psi(u)}-\psi(u)\phi'(u)\Gamma(u,\tilde \nu)
	+\frac{p}{2}\psi(u)\Gamma(u)^{\frac{p-2}{2}}\Gamma(\Gamma(u),\tilde \nu) \notag \\
	&=-p\psi(u)II(u,u)\leq0, \notag
	\end{align}
	where $\tilde \nu$ is the outward normal function at $x$. This gives $\Gamma(F,\tilde \nu)(x)=0$ and since $x$ is a maximum point, this implies that the second derivative in normal direction must be non-positive. Obviously, all the second-order derivatives in tangent directions must be non-positive as well, so the ellipticity yields   $\mathcal{L}^u_p(F)(x)\leq 0$, as desired. 
	Now the identity $\Gamma(F,u)(x)=0$ and the product and chain rule imply at $x$ that
	$$
	0
	=\psi'(u)\frac{F}{\psi(u)}\Gamma(u)+\psi(u)\frac{p}{2}\Gamma(u)^{\frac{p-2}{2}}\Gamma(\Gamma(u),u)-\psi(u)\phi'(u)\Gamma(u),
	$$
	which yields
	\begin{align}
	\label{first.der.zero}
	\Gamma (u)^{\frac{p-2}{2}}A_u=-\frac{1}{p}\bigg(\frac{\psi'}{\psi^2}F-\phi'\bigg).
	\end{align}
	Next, we would like to take a closer look at the inequality $\mathcal{L}^u_p(F)(x)\leq0$. We compute at $x$, using the diffusion property of $L$, that
	\begin{align*}
	L(F)=&\psi(u)(L(\Gamma(u)^{\frac{p}{2}})-L(\phi(u))
	+(\Gamma (u)^{\frac{p}{2}}-\phi(u))(\psi'(u)L(u)+\psi''(u)\Gamma(u))
	\\&+2\big(p\Gamma(u)^{\frac{p-2}{2}}\psi'(u)H_u(u,u)-\phi'(u)\psi'(u)\Gamma(u)\big).
	\end{align*}
	On the other hand, using the product and chain rule {Lemma \ref{chain.rule}}, we compute
	\begin{align*}
	H_F(u,u)=&(\psi'(u)H_u(u,u)+\psi''\Gamma(u)^2)(\Gamma(u)^{\frac{p}{2}}-\phi(u))
	+\psi(u)H_{\Gamma(u)^{\frac{p}{2}}}(u,u)\\&-\psi(u)H_{\phi(u)}(u,u)
	+2p\Gamma(u)^{\frac{p}{2}}\psi'(u)H_u(u,u)
	-2\psi'(u)\phi'(u)\Gamma(u)^2.
	\end{align*}
	Using these two identities as well as (\ref{first.der.zero}), the strong eigenvalue equation and $\Gamma(u)^{\frac{p}{2}}-\phi(u)=F/\psi(u)$ as well as $\Gamma(u)^{\frac{p}{2}}=F/\psi(u)
	+\phi(u)$, we obtain at $x$ that
	\begin{align}
	\mathcal{L}^u_p(F)=&p\psi(u)\frac{1}{p}\mathcal{L}^u_p(\Gamma(u)^{\frac{p}{2}}) \notag
	-\psi(u)\mathcal{L}^u_p(\phi(u))-\lambda F\frac{\psi'(u)}{\psi}u|u|^{p-2}
	\\&+(p-1)F\frac{\psi''(u)}{\psi(u)}\bigg(\frac{F}{\psi(u)}+\phi(u)\bigg)
	-2(p-1)F\frac{\psi'(u)^2}{\psi(u)^2}\bigg(\frac{F}{\psi(u)}+\phi(u)\bigg). \notag
	\end{align}
	Now the idea is to use {Corollary \ref{estimate.summary}} to estimate $\mathcal{L}^u_p(F)$ further and express the resulting inequality in terms of $w$. Since $w$ satisfies a differential equation, an appropriate choice of $\psi$ will enforce $F\leq 0$. If $n=\infty$ and $T\equiv 0$, the proof stays the same with the conventions $\infty/\infty=1$ and $1/\infty=0$. Noting that we have excluded the case $\Gamma(u)=0$ and using {Corollary \ref{estimate.summary}}, equation (\ref{first.der.zero}) as well as $\Gamma(u)^{\frac{p}{2}}=F/\psi(u)+\phi(u)$, we see that
	\begin{align}
	\frac{1}{p}\mathcal{L}^u_p(\Gamma(u)^{\frac{p}{2}})
	\geq &-\lambda(p-1)|u|^{p-2}\Gamma(u)^{\frac{p}{2}}+\lambda \frac{(n+1)(p-1)-(n-1)}{(n-1)}u|u|^{p-2}
	\Gamma(u)^{\frac{p-2}{2}}A_u \notag
	\\&+\frac{\lambda^2u^{2p-2}}{n-1}
	+\frac{n}{n-1}(p-1)^2\Gamma(u)^{\frac{2p-4}{2}}A_u^2 \notag
	\\=&\frac{\lambda^2u^{2p-2}}{n-1}+\lambda\frac{(n+1)(p-1)-(n-1)}{p(n-1)} \phi'u|u|^{p-2}
	\notag
	\\&+\frac{n}{n-1}\frac{(p-1)^2}{p^2}\phi'^2
	-\lambda \phi(p-1)|u|^{p-2} \notag
	\\&+F\bigg(-\lambda(p-1)|u|^{p-2}\frac{1}{\psi}-\lambda\frac{(n+1)(p-1)-(n-1)}{p(n-1)}\frac{\psi'}{\psi^2}u|u|^{p-2}
	\notag
	\\&-2\frac{n}{n-1}\frac{(p-1)^2}{p^2}\phi'\frac{\psi'}{\psi^2} \bigg) 
	+F^2\frac{n}{n-1}\frac{(p-1)^2}{p^2}\frac{\psi'^2}{\psi^4}. \label{applied.bochner.formula}
	\end{align}
	On the other hand, we easily verify that
	$$
	\mathcal{L}^u_p(\phi(u))=\phi'L_p(u)+(p-1)\phi''\Gamma(u)^\frac{p}{2}
	=-\lambda \phi'u|u|^{p-2}+(p-1)\phi''\phi+F(p-1)\frac{\phi''}{\psi}.
	$$
	Combining these identities and summing up the $u|u|^{p-2}$ terms, we get that
	\begin{align}
	0\geq\mathcal{L}^u_p(F)\geq&p\psi\bigg(\frac{\lambda^2u^{2p-2}}{n-1}+\lambda\frac{(n+1)(p-1)}{p(n-1)}\phi 'u|u|^{p-2}
	\notag\\&+\frac{n}{n-1}\frac{(p-1)^2}{p^2}\phi'^2
	-\lambda \phi(p-1)|u|^{p-2}-\frac{(p-1)}{p}\phi''\phi\bigg)\notag
	\\&+F\bigg(-\lambda p(p-1)|u|^{p-2}-\lambda\frac{(n+1)(p-1)}{n-1}\frac{\psi'}{\psi}u|u|^{p-2}\notag\\&-2\frac{n}{n-1}\frac{(p-1)^2}{p}\phi'\frac{\psi'}{\psi}-(p-1)\phi''
	+(p-1)\phi\bigg(\frac{\psi''}{\psi}-2\frac{\psi'^2}{\psi^2}\bigg)\bigg) \notag
	\\& +F^2 \frac{p-1}{\psi}\bigg(\frac{\psi''}{\psi}+\frac{\psi'^2}{\psi^2}\bigg(\frac{n(p-1)}{p(n-1)}-2\bigg)\bigg)
	\notag
	\\=:& a+bF+cF^2. \label{final.inequality}
	\end{align}
	The last part of the proof is similar to \cite[Theorem 4.1]{val1} and we only include it for the convenience of the reader. We consider the function $\phi(s):=w'(w^{-1}(s))^p$ and the chain rule gives
	$$
	\phi'(s)=\frac{p}{p-1}\Delta_p(w)(w^{-1}(s)), \qquad \phi''(s)=\frac{p}{p-1}\frac{(\Delta_p(w))'}{w'}(w^{-1}(s)).
	$$
	On the other hand, by our assumption there exists $t\in(a,b)$ with $w(t)=u(x)$,
	so we obtain at $x$ or $t$ respectively that
	\begin{align}
	\frac{a}{p\psi}=&\frac{\lambda^2w^{2p-2}}{n-1}+\lambda\frac{n+1}{n-1}\Delta_p(w)w|w|^{p-2}+\frac{n}{n-1}(\Delta_p(w))^2\notag\\&-\lambda (p-1)(w')^p|w|^{p-2}-(\Delta_p(w))'w'^{p-1}. \notag
	\end{align}
	Now since $T$ is one of the solutions of $T'=T^2/(n-1)$, that is, $T\equiv 0$ or $T=-(n-1)/t$, using $((w')^{p-1})'=\Delta_p w$, one directly verifies that
	\begin{align}
	\frac{a}{p\psi}=&\frac{1}{n-1}\bigg(\Delta_pw-Tw'^{p-1}+\lambda w|w|^{p-2}\bigg)\bigg(n\Delta_p w+Tw'^{p-1}\notag+\lambda w|w|^{p-2}\bigg)\\&
	-w'^{p-1}\bigg(\Delta_pw+T|w'|^{p-1}+\lambda w|w|^{p-2}\bigg)'\notag= 0.\notag
	\end{align}
	In order to treat the terms $b$ and $c$, we define
	$$
	X(t):=\lambda^{\frac{1}{p-1}}\frac{w(t)}{w'(t)},
	\qquad \psi(t)=\exp\bigg(\int_{0}^{t}h(s)ds\bigg), \qquad f(t)=-h(w(t))w'(t),
	$$ for some $h$ which is still to be determined. Using that $w$ solves (\ref{model.ode}), we compute
	\begin{align}
	f'(t)&=-h'(w(t))w'(t)^2+\frac{f(t)}{p-1}(T-X|X|^{p-2}). \label{f.der}
	\end{align}
	Now we recall that by definition
	$$
	c=\frac{p-1}{\psi}\bigg(\frac{\psi''}{\psi}+\frac{\psi'^2}{\psi^2}\big(\frac{n(p-1)}{p(n-1)}-2\big)\bigg),
	$$
	such that (\ref{f.der}) and a direct computation give
	\begin{align}
	\frac{c(w(t))\psi(w(t))}{p-1}w'(t)^2=\frac{f}{p-1}(T-X|X|^{p-2})+f^2\bigg(\frac{p-n}{p(n-1)}\bigg)-f'=:\alpha(f,t)-f'. \label{alpha}
	\end{align}
	On the other hand, we have
	\begin{align*}
	b=&-\lambda p(p-1)|u|^{p-2}-\lambda\frac{(n+1)(p-1)}{n-1}\frac{\psi'}{\psi}u|u|^{p-2}
	\\&-2\frac{n}{n-1}\frac{(p-1)^2}{p}\phi'\frac{\psi'}{\psi}-(p-1)\phi''
	+(p-1)\phi\big(\frac{\psi''}{\psi}-2\frac{\psi'^2}{\psi^2}\big).
	\end{align*}
	This gives
	\begin{align}
	\frac{b}{(p-1)|w'|^{p-2}}=&\frac{p}{p-1}T\bigg(\frac{n}{n-1}T-X|X|^{p-2}\bigg)\notag \\&-f^2+f\bigg(\bigg(\frac{2n}{n-1}+\frac{1}{p-1}\bigg)T-\frac{p}{p-1}X|X|^{p-2}\bigg)-f' \notag
	\\=&:\beta(f,t)-f' \label{beta}
	\end{align}
	which is again verified by direct computation. Now according to {Lemma \ref{lemma.grad.comp}} below, $f$  can be chosen such that $b,c>0$, that is,
	$$
	0\geq bF+cF^2\geq bF
	$$ 
	which implies $F\leq 0$ as desired (choosing $f$ rather than $h$ does not make a difference since $w$ is invertible and $w'>0$). 
\end{proof}
In the proof, we have used
\begin{lemma}
	\label{lemma.grad.comp}
	Let $\alpha,\beta$ be defined as in (\ref{alpha}),(\ref{beta}). Then for every $\epsilon>0$ there exists a smooth function $f:[a+\epsilon,b(a)-\epsilon]$ such that
	$$
	f'(t)<\min\{\alpha(f(t),t),\beta(f(t),t)\}.
	$$	
\end{lemma}
\begin{proof}
	The proof relies on properties of the model function $w$ and uses the Pr\"ufer transformation (see \cite[Lemma 5.2]{val1}). 
\end{proof}

\begin{remark}
	If $1<p<2$ and $u(x)=0$, it only follows that $u\in C^{2,\alpha}$ near $x$. If this happens, the p-Bochner formula is not directly applicable. However, since the gradient of $u$ does not vanish in an environment $U$ of $x$, the set $U':=U\cap\{u\neq 0\}$ is dense and open
	in $U$. $u$ is smooth in $U'$ and thus satisfies the strong eigenvalue equation, and hence we can replace the term  $\Gamma(u,L_p u)$ arising from $\mathcal{L}^u_p(\Gamma(u)^{\frac{p}{2}})$ by $-\lambda\Gamma(u,u|u|^{p-2})$.
	We get another diverging term $-\psi(u)\phi''(u)\Gamma(u)^{\frac{p}{2}}$ from
	$-\psi(u)\mathcal{L}^u_p(\phi(u))$, and these two terms cancel out because $\phi''(u)$ includes a $-\lambda u|u|^{p-2}$ term as well. So we have
	for $x'\in U'$ that $\mathcal{L}^u_p(F)(x')$ converges as $x'\to x$ and we also denote
	the limit by $\mathcal{L}^u_p(F)(x)$. We easily verify that $0\geq\mathcal{L}^u_p(F)(x)$ is still valid. By definition of
	$\mathcal{L}^u_p(F)(x)$, the identity (\ref{final.inequality}) 
	still holds with the two diverging terms canceled out. Now we can proceed as in the normal proof.
	\label{regularity.grad.comp}
\end{remark}
\subsection{Maximum comparison}

In this subsection, we use the gradient comparison to compare the maximum of the eigenfunctions and the model functions. Again, our approach is to generalize the idea of
\cite{val1}.
Let $u$ be an eigenfunction of $L_p$ with Neumann boundary conditions satisfying $\min u=-1$ and $\max u < 1 $. We assume that $L$ satisfies the condition BE$(0,N)$ for some $N\in[1,\infty)$ where we emphasize that we have excluded the case $N=\infty$. Let $n\geq N$ and $n>1$. By {Theorem \ref{ode.asymptotics}}, there exists a solution $w_a$ to $(\ref{model.ode})$ such that
$[\min(u),\max(u)]\subset[-1,m(a)]$ where $a\in[0,\infty)$. For ease of notation, we will write $w:=w_a$ unless specified. The differential equation implies that $w''$ stays positive until the first root of $w$, so $w$ has a unique root $t_0\in(a,b)$.
By {Theorem \ref{grad.comp.thm}}, the gradient comparison
$$
\Gamma(w^{-1}\circ u)\leq 1
$$
holds.
We will obtain the maximum comparison by comparing the volumes of small balls with respect to certain measures. In order to do that, we let $g:=w^{-1}\circ u$ and define the measure $\mu:=g_{*}m$ on $[a,b(a)]$. That is,
for any measurable function $f:[a,b]\to\mathbb{R}$ we have
$$
\int_{a}^{b}fd\mu=\int_M f\circ g dm.
$$
\indent
The first step in our volume comparison is the following theorem, which can be seen as a comparison theorem for the density of $\mu$.
\begin{theorem} \label{volume.density}
	Let $u$ and $w$ be as above and define
	$$
	E(s):=-\exp\bigg(\int_{t_0}^s \frac{w|w|^{p-2}}{w'|w'|^{p-2}}dt\bigg)\int_a^sw|w|^{p-2}d\mu.
	$$
	Then $E$ is increasing on $(a,t_0]$ and decreasing on $[t_0,b)$.
\end{theorem}
This result can also be stated in a more convenient way, as we will soon see:
\begin{theorem}
	Under the hypothesis of Theorem \ref{volume.density} the function
	$$
	E(s):=\frac{\int_a^s w|w|^{p-2}d\mu}{\int_a^s w|w|^{p-2}t^{n-1}dt}=
	\frac{\int_{u\leq w(s)}u|u|^{p-2}dm}{\int_a^s w|w|^{p-2}t^{n-1}dt}
	$$
	is increasing on $(a,t_0]$ and decreasing on $[t_0,b)$. \label{volume.density.2}
\end{theorem}
\begin{proof} See \cite[Theorem 6.3]{val1}. 
\end{proof}
\begin{proof}[Proof of {Theorem \ref{volume.density}}] 
	This is again very similar to \cite[Theorem 6.2]{val1}:
	Let $H\in C^\infty_c((a,b))$ with $H\geq 0$ and consider the ordinary differential equation
	$$
	\begin{cases}
	&\frac{\partial}{\partial t}\bigg(G(w(t))^{p-1}\bigg)=H(t),\\ &G(-1)=0.
	\end{cases}
	$$
	Since $H$ has compact support, the singularities at the boundary are avoided, so after rewriting the equation, existence follows by the Picard-Lindel\"of theorem ($w(t)$ can be seen as a coordinate change). We therefore have
	\begin{align}
	G|G|^{p-2}\circ(w(t))=\int_{a}^{t} H(s)ds, \qquad (p-1)G'(w(t))|G(w(t))|^{p-2}w'(t)=H(t). \label{g.identity}
	\end{align}
	Next, define $K$ to be an antiderivative of $G$. 
	Since $L$ is a diffusion operator, we have
	$L(K(u))=K''(u)\Gamma(u)+K'(u)L(u)$ and the chain rule gives $\Gamma(K(u))=K'(u)^2\Gamma(u)$ as well as $H_{K(u)}(K(u),K(u))=\frac12 \Gamma(K(u),\Gamma(K(u)))=K'(u)^3H_u(u,u)+K''(u)K'(u)^2\Gamma(u)^2$.
	So using that $K'=G$, we obtain at points where $u$ is in $C^{2,\alpha}$, that is, at non-critical points,
	\begin{align}
	L_p(K(u))&=G(u)|G(u)|^{p-2}L_p(u)+G'(u)|G(u)|^{p-2}\Gamma(u)^{\frac{p}{2}}\notag\\
	&=
	-\lambda u|u|^{p-2}G(u)|G(u)|^{p-2}+(p-1)G'(u)|G(u)|^{p-2}\Gamma(u)^{\frac{p}{2}}.
	\label{g.identity2}\end{align}
	Now we consider the closed (hence compact) set $C:=\{p\in M | \Gamma(K(u))=0 \}$. If $\dim(M)=1$, we can just integrate between critical points. So we assume that $\dim(M)\geq 2$ and choose a cut-off function  $\phi\in C^\infty(M)$ satisfying $0\leq\phi\leq1$,  $\phi|_{M\setminus B_{2\epsilon }(C)}\equiv 1$, $\phi|_{M\setminus B_{\epsilon }(C)}\equiv 0$, and $\Gamma(\phi)\leq2/\epsilon^2$. Using the partial integration formula {Lemma \ref{int.by.parts}} together with $\Gamma(u,\tilde \nu)\equiv 0$ and (\ref{g.identity2}), we get
	\begin{align}
	\int_M\bigg(-\lambda u|u|^{p-2}G(u)|G(u)|^{p-2}+(p-1)G'(u)|G(u)|^{p-2}\Gamma(u)^{\frac{p}{2}}\bigg)\phi \notag \\
	=-\int_M \Gamma(K(u))^{p-2}\Gamma(K(u),\phi). \label{approximating.equality}
	\end{align}
	The right-hand side of (\ref{g.identity2}) is integrable on $M$ and the second term vanishes on $C$. Since $u\in C^{1,\alpha}(M)$, $|C|\neq0$ is only possible if $C$ contains open points. Near such a point, we either have $G(u)\equiv 0$ or $\Gamma(u)\equiv 0$. In the latter case, the weak eigenvalue equation implies $u\equiv 0$ near such points, so the first term vanishes on $C$ up to a set of measure $0$, too. Hence, letting $\epsilon \to 0$, the left-hand side of (\ref{approximating.equality}) converges to 
	$$
	\int_M\bigg(-\lambda u|u|^{p-2}G(u)|G(u)|^{p-2}+(p-1)G'(u)|G(u)|^{p-2}\Gamma(u)^{\frac{p}{2}}\bigg).
	$$
	On the other hand, the right-hand side of (\ref{approximating.equality}) is identically zero on $C$, and we can estimate using the Cauchy-Schwarz inequality
	$$
	\bigg|\int_{B_{2\epsilon}(C)\setminus B_\epsilon(C)} \Gamma(K(u))^{p-2}\Gamma(K(u),\phi)\bigg|\leq |B_{2\epsilon}(C)|\frac{A}{\epsilon}\to 0
	$$
	since $\Gamma(K(u))$ is uniformally bounded and $|B_{2\epsilon}(C)|\leq A'\epsilon ^2$. Here, the last statement is  implied by the Bishop-Gromov volume growth theorem, valid for the metric measure space $(M,d,m)$ with the dimension upper bound $N\geq\dim(M)\geq2$ (see \cite[Theorem 4]{sturm2}, an elliptic operator cannot satisfy a condition BE$(0,N')$ for $N'<\dim (M)$, which is easily seen since this would also imply $\operatorname{tr}H=L$). Since $\Gamma(\phi)\equiv 0$ on $M\setminus B_{2\epsilon}(C)$, we finally obtain 
	$$
	\int_M\lambda u|u|^{p-2}G(u)|G(u)|^{p-2}=\int_M (p-1)G'(u)|G(u)|^{p-2}\Gamma(u)^{\frac{p}{2}}.
	$$
	Now using the gradient comparison {Theorem \ref{grad.comp.thm}} and the definition of $\mu$, we get
	\begin{align*}
	\frac{\lambda}{p-1}\int_{a}^{b} w(t)|w(t)|^{p-2}G(w(t))|G(w(t))|^{p-2}d\mu
	&\leq \int_{a}^{b}G'(w(t))|G(w(t))|^{p-2}|w'(t)|^pd\mu.
	\end{align*}
	With the identities (\ref{g.identity}), this reads
	\begin{align}
	\lambda\int_{a}^{b} w(t)|w(t)|^{p-2}\int_{a}^{t}H(s)dsd\mu\leq\int_{a}^{b}H(t)w'(t)^{p-1}d\mu.
	\label{w.inequality}
	\end{align}
	On the other hand, using Fubini's theorem and $\int_{a}^{b}w(t)|w(t)|^{p-2}=0$, we get
	\begin{align*}
	\lambda\int_{a}^{b} w(t)|w(t)|^{p-2}\int_{a}^{t}H(s)dsd\mu&=\lambda\int_{a}^{b}\int_{s}^{b}w(t)|w(t)|^{p-2}d\mu H(s)ds
	\\&=-\lambda\int_{a}^{b}\int_{a}^{s}w(t)|w(t)|^{p-2}d\mu H(s)ds.
	\end{align*}
	Combining this with (\ref{w.inequality}), we obtain
	$$
	\int_{a}^{b}\bigg(-\lambda\int_{a}^{s}w(t)|w(t)|^{p-2}d\mu\bigg) H(s)ds\leq\int_{a}^{b}H(s)w'(s)^{p-1}d\mu.
	$$
	Since this is valid for any non-negative function $H\in C^\infty_c((a,b))$, we get 
	$$
	w'(s)^{p-1}g(s)+\lambda\int_{a}^{s}w(t)|w(t)|^{p-2}d\mu\geq 0
	$$
	for any $s\in(a,b)$, where $g$ is the density of $\mu$. Multiplying by $w|w|^{p-2}/(w'|w'|^{p-2})$, we get
	$$
	w|w|^{p-2}(s)g(s)+\frac{w|w|^{p-2}}{w'|w'|^{p-2}}(s)\lambda\int_{a}^{s}w(t)|w(t)|^{p-2}d\mu\begin{cases} &\geq 0 \text{ if } s\leq t_0,
	\\&\leq 0 \text{ if } s>t_0.
	\end{cases}
	$$
	But this is exactly the derivative of $E$, so the theorem is proven. 
\end{proof}
In order to prove the maximum comparison, we want to compare the volumes of small balls near critical points. Therefore, we need the following lemma:
\begin{lemma}
	For $\epsilon$ sufficiently small, the set $u^{-1}[-1,-1+\epsilon)$ contains
	a ball with radius $r_\epsilon$, where
	$$
	r_\epsilon=w^{-1}(-1+\epsilon)-a.
	$$
	\label{ball.lemma}
\end{lemma}
\begin{proof}
	Let $x_0\in M $ be a minimum point, that is, $u(x_0)=-1$ and $x\in M$ be
	another point. By the gradient comparison, we have $\Gamma(w^{-1}(u))\leq 1$, so by the
	definition of the distance function
	$$
	d(x,x_0)\geq |w^{-1}\circ u|_{x_0}^x|=w^{-1}(u(x))-a.
	$$
	So if $d(x,x_0)<r_\epsilon$ then
	$$
	w^{-1}(u(x))<w^{-1}(-1+\epsilon)
	$$
	and since $w$ is increasing, we must have $u(x)<-1+\epsilon$, which proves the claim. 
\end{proof}
Now we are able to prove the following volume comparison:
\begin{theorem}[volume comparison] Let $n\geq N$ and $n>1$.
	If $u$ is an eigenfunction satisfying $\operatorname{min}u=-1=u(x_0)$ and
	$\operatorname{max}u\leq m(0)=w_0(b(0))$, then there exists a constant $c>0$ such that for all
	$r$ sufficiently small we have
	$$
	m(B_{x_0}(r))\leq cr^n.
	$$
	\label{vol.comparison}
\end{theorem}
\begin{proof}
	This proof is in spirit of \cite[Theorem 6.5]{val1}. We define the measure $\gamma:=t^{n-1}dt$ on $[0,\infty)$. Let $\epsilon>0$ be small enough such that {Lemma \ref{ball.lemma}} is applicable and that $-1+\epsilon \leq -1/2^{\frac{1}{p-1}}$. Hence, for $u\leq -1+\epsilon$ we have $-u|u|^{p-2}\geq 1/2$. Also, the point in time $t$, where $w_0(t)=-1+\epsilon$, occurs before the first zero of $w_0$, so {Theorem \ref{volume.density.2}} implies $E(t)\leq E(t_0)=:C$. Multiplying this inequality by $-\int_{a}^{s}w_0|w_0|^{p-2}d\gamma>0$, we get
	\begin{align*}
	\operatorname{Vol}(\{u\leq -1+\epsilon\})&\leq -2\int_{\{u\leq -1+\epsilon\}}u|u|^{p-2}dm\\&\leq-2C\int_{\{w_0\leq -1+\epsilon\}}w_0|w_0|^{p-1}d\gamma\\&\leq 2C'\gamma(\{w_0\leq -1+\epsilon\}).
	\end{align*}
	On the other hand, with $r_\epsilon$ from {Lemma \ref{ball.lemma}}, we get
	$$
	\operatorname{Vol}(B_{r_\epsilon}(x_0))\leq\operatorname{Vol}\{u\leq-1+\epsilon\}\leq2C\gamma(\{w_0\leq -1+\epsilon\})=2C\nu([0,r_\epsilon])=C'r^n_\epsilon
	$$ 
	since for $a=0$, we have $w_0(r_\epsilon)=-1+\epsilon$. 
\end{proof}
This finally allows us to prove the desired maximum comparison which will always provide a suitable model function:
\begin{corollary}
	Let $n\geq N$, $n>1$, and $w_0$ be the corresponding model function.
	If $u$ is an eigenfunction satisfying $\operatorname{min}u=-1=u(x_0)$ ,
	then
	$\operatorname{max}u\geq m(0)$. \label{max.comparison}
\end{corollary}
\begin{proof}
	This is obvious if $\max(u)=1$, so we can assume that $\max(u)<1$. Assuming the assertion were wrong, then by the continuous dependence of 
	the data there would exist a solution of the differential equation (\ref{model.ode}) with the same $\lambda$, $a=0$, and $n'>n$, in particular $n'>N$,
	whose first maximum would still be bigger than $\operatorname{max} u$. By 
	{Lemma \ref{improved.BE}}, the gradient comparison would still hold and so would the volume comparison {Theorem \ref{vol.comparison}}, that is, for $r$ sufficiently small,
	$$
	m(B_{x_0}(r))\leq cr^{n'}.
	$$
	This, however, contradicts the Bishop-Gromov theorem (see \cite[Theorem 4]{sturm2}): since the metric measure space $(M,d,m)$ satisfies $\operatorname{BE}(0,N)$ and thus also $\operatorname{CD}^e(0,N)$, we get for small $r$ and some constant $c'>0$
	$$
	m(B_{x_0}(r))\geq c'r^{N}
	$$
	which contradicts the first estimate for small values of $r$. 
\end{proof} 
\subsection{Sharp estimate}
\label{sharp.estimate}
We can now combine the estimates for gradient and maximum with the theory of the one-dimensional model. As we have seen, in the one-dimensional case with $L=\Delta$, the first eigenvalue is $(p-1)\pi^p/D^p$ where $D$ is the diameter, so the next result is sharp.
\begin{theorem}
	Let $M$ be compact and connected and $L$ be an elliptic diffusion operator with invariant measure $m$. We assume as well that $L$ satisfies $\operatorname{BE}(0,N)$ for some $N\in[1,\infty)$ and if the boundary is non-empty, we assume it to be convex. Let $\lambda$ be the principal eigenvalue of $L_p$ with Neumann boundary conditions. Then we have the sharp estimate
	\begin{align}
	\lambda\geq (p-1)\frac{\pi_p^p}{D^p}, \label{sharp.estimate.stated}
	\end{align}
	where $D$ is the diameter associated with the intrinsic metric $d$.
	\label{sharp.estimate.thm}
\end{theorem}
\begin{proof}
	Let $u$ be the rescaled eigenfunction such that $\min u=-1$ and
	$\max u\leq 1$. By {Corollary \ref{max.comparison}}, we have $m(0)\leq\max u$. On the other hand, according to {Theorem \ref{ode.asymptotics}}, $m(a)$ is a continuous function  with $m(a)\to 1$ as $a\to\infty$. Hence, there is a unique $a\in[0,\infty]$ satisfying
	$m(a)=\max u$. Let $w_a$ be the corresponding solution. The gradient comparison 
	gives $\Gamma (w_a^{-1}\circ u)\leq 1$. Let $x,y$ be maximum and minimum points of $u$, respectively. By definition of $D$ and {Theorem \ref{ode.asymptotics}}, we have
	$$ 
	D\geq |w_a^{-1}\circ u(x)-w_a^{-1}\circ u(y)|=w_a^{-1}(m(a))-w_a^{-1}(-1)=b(a)-a\geq \frac{\pi_p}{\alpha}.
	$$
	Since $\delta_a=\pi_p/\alpha$ if and only if $a=\infty$, we obtain that $\max(u)=1$ is a necessary, but, as we will see, not sufficient condition for equality to hold. 
\end{proof}
\begin{remark}
	If $L$ only satisfies the condition BE$(0,\infty)$, the situation is slightly different:
	the maximum comparison does not hold anymore, since we cannot compare the gradient of the eigenfunction with the relaxed model function. However, applying the gradient comparison with $w_\infty$ and using the symmetry of $w_\infty$ we easily obtain the estimate 
	$$
	\lambda \geq (p-1)\frac{1}{2^p} \frac{\pi^p_p}{D^p}.
	$$
	which we expect not to be sharp.
\end{remark}
We now turn towards the case of equality. Once again, motivated by the approach in \cite{val1}, we prove the following necessary condition:
\begin{theorem} \label{equality.statement}
	Under the hypothesis of Theorem \ref{sharp.estimate.thm}, we assume that equality holds in the estimate (\ref{sharp.estimate.stated}) and that $u$ is an eigenfunction with $-\min u=\max u=1$. Then the function $e(u)=\Gamma(u)^\frac{p}{2}+\lambda/(p-1)|u|^p$ is constant and equals $\lambda/(p-1)$, in particular, $\Gamma(u)=0$ if and only if $|u|=1$. Furthermore, we have that $R\equiv 0$ and $\dim(M)=1$.. 
\end{theorem} 
\begin{proof}
	In the one-dimensional model, we have the identity
	$$
	w'^p(t)+\frac{\lambda}{p-1}|w|^p(t)\equiv \frac{\lambda}{p-1},
	$$
	which is readily checked by differentiating and  testing the equation at one of the endpoints. Hence, the gradient comparison gives
	$$
	\Gamma(u)^{\frac{p}{2}}\leq w'(w^{-1}(u))^p=\frac{\lambda}{p-1}\bigg(1-|u|^p\bigg),
	$$
	so we have $e\leq \lambda/(p-1)$. Let $x,y\in M$ be minimal and maximal points of $u$ respectively. We have
	$$
	D\geq d(x,y)=\Gamma(w^{-1}\circ u)|^y_x= \frac{\pi_p}{\alpha}=D
	$$
	by the equality assumption. Hence, the distance between $x,y$ is attained by $w^{-1}\circ u$. Therefore, we must have $\Gamma(w^{-1}\circ u)(z)=1$ for some $z\in M\setminus \{|u|=1\}$, because otherwise the function cannot attain the supremum in the definition of the intrinsic distance. Now we would like to use the strong maximum principle: we consider the operator
	\begin{align*}
	\mathcal{P}(\eta)=&\mathcal{L}^u_p(\eta)-(p-2)\lambda u|u|^{p-2}\frac{\Gamma(u,\eta)}{\Gamma(u)}\\&+(p-1)^2\Gamma(u)^{\frac{p-4}{2}}\bigg(H_u(u,\phi)-A_u\Gamma(u,\phi)\bigg)
	-\frac{(p-1)^2}{p\Gamma(u)}\Gamma\big(\Gamma(u)^{\frac{p}{2}}-\frac{\lambda}{p-1}|u|^p,\phi\big)
	\\=&:\mathcal{L}^u_p(\eta)+\mathcal{P}_0(\eta)
	\end{align*}
	and notice that $\mathcal{P}$ is locally uniformally elliptic in the open set $M\setminus \{\Gamma(u)=0\}$. The first-order term $\mathcal{P}_0$ is made in a way such that $\mathcal{P}(e)\geq 0$. Indeed, the chain rule gives
	$\mathcal{L}^u_p(|u|^p)=p\big(u|u|^{p-2}L_p(u)+(p-1)^2|u|^{p-2}\Gamma(u)^p\big)$ and one easily verifies that $A_u(u,u)\Gamma(u)=H_u(u,u)$. Combining this with
	the $p-$Bochner formula, the strong eigenvalue equation, some algebraic manipulations as well as the identity $p(p-2)-(p-1)^2=-1$, we obtain 
	
	$$
	\mathcal{P}(e)=p\Gamma(u)^{p-2}\bigg(R(u,u)+|H_u|^2_{HS}-A_u^2\bigg)\geq 0
	$$ by the definition of the Hilbert-Schmidt norm. Now the strong maximum principle (\cite[Theorem 3.5, Chapter 3]{gilbarg}) gives that the set $\{e=\lambda/(p-1)\}$ is open and closed in $M\setminus(\{|u|=1\}\cup\{\Gamma(u)=0\})=:M\setminus E$, in particular, $e\equiv 1$ in the component containing $z$, which we denote by $C_0$. If there is another component $C_1$, we let $p'\in C_1$ such that $u(p')=0$. Similarly as above, we have that
	$$
	\operatorname{dist}(p,E)\geq |(w^{-1}(\pm 1))-(w^{-1}(0))|=\frac 12 \frac{\pi_p}{\alpha}=\frac12D
	$$ 
	by the equality assumption. Now by the intermediate value theorem, we have
	$$
	D\geq\operatorname{dist}(\{u=1\},\{u=-1\})\geq \operatorname{dist}(\{u=1\},\{u=0\})+
	\operatorname{dist}(\{u=0\},\{u=-1\})\geq D.
	$$
	In particular, the function $w^{-1}\circ u$ attains the maximum in the definition of the distance to some boundary point, so we obtain that there is a point $z'\in C_1$ with $\Gamma(w^{-1}\circ u)(z')=1$. Hence, $e\equiv \lambda/(p-1)$ in $C_1$, and since this holds on $E$ anyway, we get that $e\equiv \lambda/(p-1)$ on $M$. Now by the regular value theorem, the level sets $\{u=t\}$ are smooth submanifolds of dimension $\dim(M)-1$ for any$|t|<1$, and so are the sets $D_t:=C_1\cap\{u=t)\}$. In order to get a useful frame, we define the map $\Phi:D_0\times(-D/2,D/2)\to M $ by
	\begin{align*}
	&\Phi(x,0)=x,
	\\ \frac{\partial}{\partial t} &u(\Phi(x,t))=1.
	\end{align*}
	This is well-defined by the standard theory for ordinary differential equation since $u$ is regular in $D_0$, and because the differential equation enforces $\Phi(x,t)\in D_t$ so the solution cannot blow up. In particular, we get $\operatorname{Im}(\Phi)\subset C_1$. Smooth dependence on the data implies that $\Phi$ is smooth. Uniqueness gives that $\Phi$ is a bijection onto $C_1$: surjectivity follows, since we can solve the differential equation backwards from a point $x\in D_t$, and if  $\Phi(x,t)=\Phi(y,t')$, then we first have $t=t'$ and also $x=y$ because otherwise we could solve the differential equation backwards and obtain two distinct solutions starting at $\Phi(x,t)$ contradicting uniqueness. Now we would like to use the parametrization $\Phi$ to get a better understanding of the geometry of $C_1$: given a smooth function $\tilde v$, we define $v(t,x):=\tilde v(0,x)$ and note that $v$ is smooth, too.
	Since the differential of a function is perpendicular to its level sets we have 
	$$
	\Gamma(u,v)(x,t)=0.
	$$
	We remark that this also implies $\Gamma(u,\cdot)=\Gamma(u)\frac{\partial}{\partial t}$ since $\frac{\partial}{\partial t}u=1$ in the chosen coordinate frame.  Now the important observation is that since $\mathcal{P}(e)\equiv 0$, we have $A_u^2=|H_u|^2_{HS}$ and this implies $H_u(a,b)=\eta\Gamma(u,a)\Gamma(u,b)$ for a smooth function $\eta$, which we do not need to determine. This gives 
	\begin{align*}
	\frac{1}{\Gamma(u)}\frac{\partial}{\partial t}\Gamma(v)(x,t)&=\Gamma(u,\Gamma(v))=\Gamma(u,\Gamma(v))-2\Gamma(v,\Gamma(u,v))\\&=-2H_u(v,v)=-2\eta\Gamma(u,v)^2=0,
	\end{align*} 
	where we used that $\Gamma(u,v)$ identically vanishes on the level sets of $u$. This implies that $\Gamma(v)(x,t)=\Gamma(x,0)$ which forces $M$ to be one-dimensional: if we assume that $D_0$ has more than two points, say $x$ and $y$, then we can find a smooth function $\tilde v$ with $\tilde v(x,0)-\tilde v(y,0)>0$ and $\Gamma(\tilde v)\leq 1$. We define $v(x,t):=\tilde v(x,0)$ and rescale by a constant $c$, such that $\Gamma(v)|_{D_0}\leq 1$. Then we  have $v(x,t)-v(y,-t)=c(\tilde v(x,0)-\tilde v(y,0)>0$ for any $t\in[0,D/2)$ and the above consideration gives that $\Gamma(v)\leq 1$ on $D_0\times (-D/2,D/2)$. We have also seen that $\Gamma(u,v)=0$, so using the chain rule, we can directly compute 
	\begin{align*}
	\Gamma(\sqrt{((w^{-1}\circ u)^2+v^2)})&=\frac{1}{(w^{-1}\circ u)^2+v^2}\Gamma((w^{-1}\circ u)^2+v^2)\\&\leq \frac{1}{(w^{-1}\circ u)^2+v^2}\bigg((w^{-1}\circ u)^2\Gamma((w^{-1}\circ u))+v^2\Gamma(v)\bigg)\\&\leq 1.
	\end{align*}
	If we define $\bar v(x,t):=v(x,t)-v(y,-t)$, we get that $\Gamma(\bar v)=\Gamma(v)$ and similarly for $w^{-1}\circ u$, so we have 
	$$d(\Phi((x,t)),\Phi((y,-t))^2\geq(v(x,t)-v(y,-t))^2+(w^{-1}(t)-w^{-1}(-t))^2>D^2 $$ for $t$ sufficiently close to $D/2$, a contradiction. Thus, it follows that the level sets are discrete and hence $\dim(M)=1$. On the other hand, we have $\mathcal{P}(e)=0$ since $e$ is constant, so it follows that $R(u,u)\equiv 0$ and since $M$ is one-dimensional, we have that $R\equiv 0$.  \end{proof}
\indent
Finally, we would like to check if the necessary conditions derived in {Theorem \ref{equality.statement}} are sufficient. We consider the one-dimensional manifold with boundary $M:=[-D/2,D/2]$ and a general diffusion operator $L(u)=\Delta_g u+bu'$ for some smooth function $b$ and metric $g$. Since all Riemannian manifolds in one-dimension with the same diameter are isometric, we can assume that $g$ is the Euclidean metric, hence $Lu=u''+bu'$ and $\Gamma(u)=u'u'$. Let $\lambda$ be the principal eigenvalue of $L_p$ and assume that equality holds in the eigenvalue estimate (\ref{sharp.estimate.stated}), that is, $\lambda$ coincides with the principal eigenvalue of $\Delta_p$.  Let $u$ be the first eigenfunction of $L_p$ and $w$ be the first eigenfunction of $\Delta_p$ on $[-D/2,D/2]$. Theorem \ref{equality.statement} implies that equality holds in the gradient comparison, that is, 
$$\Gamma(w^{-1}\circ u)=1$$
or equivalently
$$
|u'|^p=\Gamma(u)^{\frac{p}{2}}=|w'(w^{-1})(u)|^p=\bigg(1-|u|^p\bigg).
$$
But this ODE is also solved by $w$ and $w$ satisfies the same boundary conditions as $u$ at $-D/2$ which implies $w\equiv u$. On the other hand, we have $$\Delta_pw=-\lambda w|w|^{p-2}=-\lambda u|u|^{p-2}=L_pu=\Delta_pu+\Gamma(u)^{\frac{p-2}{2}}bu'=\Delta_pw+\Gamma(w)^{\frac{p-2}{2}}bw'$$
which implies $b\equiv 0$. Hence equality is attained by all Laplace Beltrami operators in one-dimension or equivalently, all operators which satisfy BE$(0,1)$.
If $M$ does not have a boundary, we can proceed in a similar fashion so we have proven Theorem 1.1. \\
\indent
We end this section by demonstrating that although equality can only be attained for $\dim(M)=1$, the estimate is still sharp if we restrict ourselves to an arbitrary integer dimension. Let $N\in\mathbb{N}$ with $N\geq 2$ and $D>0$. The idea is to construct a thin tube which collapses to the one-dimensional model space. Precisely, we choose $D'$, such that $\pi D'<D$ and define the product manifold $M=S^1\times S^{N-1}$ with metric
$g:=D'g_{S^1}+ag_{S^{N-1}}$. We choose $L=\Delta_g$, which means that $R=\operatorname{ric}_M=\operatorname{ric}_{S^{N-1}}=(N-2)\frac{1}{a}g_{S^{N-2}}$, in particular, $R\geq 0$. Hence, $L$ satisfies BE$(0,N)$, but does not satisfy BE$(0,N')$ for any $N'<N$. Now let $w_{D'}$ be the first eigenfunction of $(\Delta_{\tilde g})_p$ on $S^1$ with metric $\tilde g:=D'{g_{S^1}}$. Let $\lambda_{D'}$ be the eigenvalue of $w_{D'}$, then we have
$$
\lambda_D'=(p-1)\frac{\pi^p_p}{(\pi D')^p}.
$$
Since the diameter depends continuously on $a$ we can chose $a=a(D')$ in a way such that $\operatorname{diam}(M)=D$. If we define $u_{D'}(t,x):=w_{D'}(t)$, then $u_{D'}$ is an eigenfunction of $(\Delta_g)_p$ with eigenvalue $\lambda_D'$. If we let $\pi D'\nearrow D$, then $\lambda_D'\to (p-1)\pi^p_p/D^p$ so the estimate cannot be improved. 

\section{Non-symmetric operators}
\label{nonsymmetric.operators}
In this section, we extend our methods to non-symmetric diffusion operators, that is, operators without an invariant measure, more precisely, we prove Theorem \ref{maintheorem1}. We restrict ourselves to the linear case $p=2$ as it does not seem like our approach could be generalized to the non-linear case. \\
\indent
We consider a smooth manifold $M$ with $\dim(M)=N$ and an elliptic diffusion operator $L$ which satisfies $\operatorname{BE}(a,\infty)$ for some $a\in\mathbb R$. We equip $T^*M$ with the metric $\Gamma$ and use the distance function $d$ induced by $L$, and $(M,g)$ thus becomes a Riemannian manifold, where $g$ is the metric on $TM$ coming from the metric $\Gamma$ on $T^*M$. As described in {Section \ref{preliminaries}}, using the metric $g$ we can write $L=\Delta_g+X\cdot\nabla$ for a suitable vector field $X$. We consider the Neumann eigenvalue problem
\begin{align}
\begin{cases}
& Lu=\lambda u\text{ on } M\\ \label{nonsymmetric.ev.eqn}
& \Gamma (u,\tilde\nu)=0 \text{ on } \partial M
\end{cases} 
\end{align}
where we require $\partial M$ to be strictly convex. We emphasize that contrary to \cite{andrews}, $X$ does not have to be the gradient of a function and hence $L$ might not possess and invariant measure.  $N$ now denotes the extrinsic dimension of $L$ which at least coincides with the intrinsic dimension of $\Delta_g$.\\ \indent
Eigenvalues of non-symmetric operators with Neumann boundary conditions can be shown to exist by standard methods but apart from the trivial eigenvalue $\lambda=0$,  they  are generally complex (see for instance \cite[Theorem 3.2, Chapter 3, Section 3]{lady}). Still, the standard Schauder-theory gives smoothness of the eigenfunctions. \\ \indent
Again, we will compare the principal eigenvalues of the operator and a one-dimensional model space. Since the principal eigenvalue of the model space is hard to compute, the result in Theorem \ref{maintheorem1} is not sharp. By using the principal eigenvalue rather than $\pi^2/D^2+a/2$ as a lower bound, it becomes sharp but less useful. Furthermore, the lower bound $\pi^2/D^2+a/2$
is the best among all linear functions in $a$ (see \cite{andrews}), which is enough for most applications.
\subsection{Modulus of continuity comparison principle}
Similar to \cite{andrews}, we show a comparison theorem for the decay of a heat equation with drift. Since every eigenfunction of $L$ with eigenvalue $\lambda$ corresponds to a solution of a heat equation with decay-rate $\operatorname{Re}(\lambda)$, this will be a suitable eigenvalue comparison, too. For the next theorem, we define the operator $\tilde L$ on $M\times M$ by $\tilde L= L_x +L_y$, where $L_x,L_y$ act on the first or second component, respectively. This also induces a metric $\tilde \Gamma =\Gamma_x+\Gamma_y$. For the first-order vector field we get the decomposition
$\tilde X=X_x+X_y$. We recall that given a metric space $M$ with diameter $D$ and distance $d$ a continuous function $\varphi:[0,D/2]\to \mathbb{R}_+$ is called a
\textit{modulus of continuity} of a function $u:M\to \mathbb{R}$ if 
$|u(x)-u(y)|\leq 2\varphi(d(x,y)/2)$ for any $x,y\in M$.
\begin{theorem} \label{modulus.of.continuity.thm}
	Let $(M,L)$ be as above with diameter $D$ and  $v$ be a smooth solution of the heat equation with drift
	$$
	\begin{cases}
	&\frac{\partial v}{\partial t}=L(v) \text{ on } M,\\
	& \Gamma(v,\tilde{\nu})=0 \text{ on } \partial M.
	\end{cases}.
	$$
	Assume further that there exists a smooth function $\phi(s,t):[0,D/2]\times\mathbb{R}_+\to\mathbb R$ such that
	\begin{itemize}
		\item[(i)] $\phi(\cdot,0)$ is a modulus of continuity for $v(\cdot,0)$,
		\item[(ii)]$\frac{\partial \phi}{\partial t}\geq\phi''-as\phi'$,
		\item[(iii)] $\phi'>0$,
		\item[(iv)] $\phi(0,\cdot)\geq 0. $
	\end{itemize}
	Then $\phi(\cdot,t)$ is also a modulus of continuity of $v(\cdot,t)$ for any $t\geq 0$.
\end{theorem}
\begin{proof}
	The idea of the proof stems from \cite[Proposition 1.1]{andrews}. Let $\epsilon>0$ and define an evolving function on $M\times M\times\mathbb{R}_+$ by
	$$
	\Phi(x,y,t)=v(y,t)-v(x,t)-2\phi\bigg(\frac{d(x,y)}{2},t\bigg)-\epsilon e^t.
	$$ 
	By letting $\epsilon\to0$ it remains to show that $\Phi$ stays negative for any choice of $\epsilon$.
	By assumption $(i)$, we have that $\Phi(x,y,0)\leq-\epsilon<0$. Assume the assertion were wrong, then there exists an $\epsilon>0$ and a first time $t_0$ such that the function attains the value $0$, say in $(x_0,y_0)\in M\times M$. Hence, the function $\Phi(\cdot,\cdot,t_0)$ attains its global maximum in $(x_0,y_0)$. Assumption ($iv$) implies that $x_0\neq y_0$. If for instance $y_0\in\partial M$ and $\tilde\nu$ is the outward normal function of $\partial M$ at $y_0$, then
	$$
	\Gamma_y(\Phi(\cdot,\cdot,t_0),\tilde\nu)(x_0,y_0)=\Gamma_y(v(y_0,t_0),\tilde\nu)-\phi'\bigg(\frac{d(x_0,y_0)}{2},t_0\bigg)\Gamma_y(d(x_0,y_0),\tilde\nu)<0,
	$$ 
	where we used the Neumann condition, $(iii)$, and that strict convexity implies that geodesics touching the boundary are outward pointing. This contradicts the maximum assumption, so we can assume that $x_0,y_0$ both lie in the interior. Now we define the functions $f(x,y):=2\phi(d(x,y)/2,t_0)$ and $\psi(x,y):=v(y,t_0)-v(x,t_0)-\epsilon e^t$. $\psi$ is smooth and touches $f$ in $(x_0,y_0)$ from below by assumption. Furthermore, we have $(\Delta_g)_x(u)=\operatorname{tr}(H_u)$, where the trace is taken with respect to $\Gamma_x$ and similarly for $y$. So in the situation of {Lemma \ref{viscosity.solution.thm}}, we have for any admissible $A\in\mathcal A$ that $\operatorname{tr}(AH_\psi)(x_0,y_0)=\Delta_g v(y_0,t_0)-\Delta_g v(x_0,t_0)$,  in particular, $\mathcal{L}(H_\psi)(x_0,y_0)=\Delta_g v(y_0,t_0)-\Delta_g v(x_0,t_0)$. Hence, {Lemma \ref{viscosity.solution.thm}} below implies that 
	\begin{align}
	\frac{\partial}{\partial t}v(y_0,t_0)-
	\frac{\partial}{\partial t}v(x_0,t_0)&=\bigg(\Delta_g v+X_y\cdot\nabla v\bigg)(y_0,t_0)-\bigg(\Delta_g v+X_x\cdot\nabla v\bigg)(x_0,t_0) \notag \\
	&=
	\bigg(\mathcal{L}(H_\psi)+X\cdot\nabla\psi\bigg)(x_0,y_0,t_0)
	\notag \\ & \leq 2\phi''\bigg(\frac{d(x_0,y_0)}{2},t_0\bigg)-a\frac{d(x_0,y_0)}{2}2\phi'\bigg(\frac{d(x_0,y_0,t_0)}{2}\bigg)\notag \\&\leq2 \frac{\partial}{\partial t}\phi\bigg(\frac{d(x_0,y_0)}{2},t_0\bigg),
	\notag
	\end{align}
	where we used $(ii)$. Hence, we obtain
	$$
	\frac{\partial}{\partial t}\Phi(x_0,y_0,t_0)=\frac{\partial}{\partial t}v(y_0,t_0)-
	\frac{\partial}{\partial t}v(x_0,t_0)-2 \frac{\partial}{\partial t}\phi\bigg(\frac{d(x_0,y_0)}{2},t_0\bigg)-\epsilon e^t<0,
	$$
	where the strict inequality is achieved by discarding the negative term $-\epsilon e^t$. This, however, contradicts the fact that $t_0$ was the first time for $\Phi$ to become zero, the assertion follows. 
\end{proof}
We add the technical Lemma needed to obtain the contradiction in the previous theorem:
\begin{lemma}
	\label{viscosity.solution.thm}
	Let $(M,L)$ be a compact and connected manifold without boundary with a possibly non-symmetric diffusion operator $L=\Delta_g+X\cdot\nabla$ satisfying $\operatorname{BE}(a,\infty)$. Let $N=\dim(M)$, $d$ be the distance function induced by $L$, and $D$ the diameter of $M$. Let $\phi:[0,D/2]\to\mathbb{R}$ be a smooth and increasing function and define the function $$f:=(M\times M)\setminus\Delta, (x,y)\mapsto 2\phi\bigg(\frac{d(x,y)}{2}\bigg).$$ Then $f$ is a viscosity supersolution of
	$$
	\mathcal{L}(H_f)+(X_x+X_y)(f)=2\phi''\bigg(\frac{d(x,y)}{2}\bigg)-ad(x,y)\phi'\bigg(\frac{d(x,y)}{2}\bigg),
	$$ 	
	where for $B\in\operatorname{Sym}(T_{(x,y)}(M\times M))$
	$$
	\mathcal{L}(B)=\inf_{A\in\mathcal A} \operatorname{tr}(AB)
	$$
	with $\mathcal{A}:=\{A\in\operatorname{Sym}(T^*_{(x,y)}(M\times M)) | A\geq 0, A|_{T^*_xM}=\Gamma_x, A|_{T^*_yM}=\Gamma_y\}$.
\end{lemma}
\begin{proof} This is a variation of the argument in \cite[Theorem 3]{julie}: Let $(x,y)\in M\times M\setminus \Delta $ and $\psi$ be a smooth function around $(x,y)$ with $\psi\leq f$ and $\psi(x,y)=f(x,y)$. $M$ is compact, so the Hopf-Rinow theorem implies that $(M,g)$ is complete and we can choose a length-minimizing geodesic $\gamma$ parametrized by the arc-length joining $x$ and $y$, that is, $\gamma(-d/2)=x$ and $\gamma(d/2)=y$, where $d:=d(x,y)$. Next, we define $e_N(s):=\gamma(s)$ and extend it to an orthonormal base $e_i$ of $T_xM$. We use parallel transport along $\gamma$ to get an orthonormal base $e_i(s)\in T_{\gamma(s)}M$ and denote $\tilde e_i:=e_i(d/2)\in T_yM$. Before defining a suitable matrix $A\in\mathcal A$, we compute the derivatives of $\psi$. Since $\phi$ is increasing, we have
	$$
	\psi(\gamma(s),\gamma(t))\leq 2\phi\bigg(\frac{d(\gamma(s),\gamma(t))}{2}\bigg)
	\leq 2\phi\bigg(\frac{|t-s|}{2}\bigg)
	$$
	with equality if $t=d/2$ and $s=-d/2$. This gives $\partial_{e_N}\psi(x,y)=-\phi'(d/2)$ and $\partial_{\tilde e_N}\psi(x,y)=\phi'(d/2)$. Now we define the smooth family of paths $\gamma^i(r,s):=\exp_{\gamma(s)}(r(\frac{1}{2}+\frac{s}{d})e_i(s))$ starting at $x$. Again, since $\phi$ is increasing, we have
	$$
	\psi(x,\exp_y(r\tilde e_i))\leq 2\phi\bigg(\frac{L(\gamma^i(r,\cdot))}{2}\bigg)
	$$
	with equality if $r=0$. The right-hand side is a smooth function of $r$ and $\gamma^i$ is a variation of the minimizing geodesic $\gamma$ which is fixed at $x$ and orthogonal at $y$. Hence, the first variation formula gives that the right-hand side has derivative zero, which implies that $\partial_{\tilde e_i}\psi(x,y)=0$  and similarly $\partial_{ e_i}\psi(x,y)=0$ for $1\leq i\leq N-1$. 
	So if we define $r(y)=d(x,y)$, we can already compute 
	\begin{align}
	(X_y+X_x)(\psi)&=\phi'(d/2)\bigg(g(X(y),\gamma'(d/2))-g(X(x),\gamma'(-d/2))\bigg)\notag\\
	&=\phi'(d/2)\bigg(\int_{-d/2}^{d/2}(g(X(\gamma(s)),\gamma'(s)))'ds\bigg) \notag \\
	&=-\phi'(d/2)\bigg(\int_{-d/2}^{d/2}\bigg(\frac12X(\Gamma_y(r,r))-\Gamma_y(X(r),r)\bigg)\circ\gamma(s)ds\bigg). \label{visc1}
	\end{align}
	To see why the last equality holds, we first note that any $\gamma(s)$ with $s\in(-d/2,d/2)$ is outside of the cut-locus of $x$ because $\gamma$ is length-minimizing. $r$ is smooth  near such points and satisfies $\Gamma_y(r,r)\equiv 1$,  which implies $X(\Gamma_y(r,r))=0$. Now for any $s$, we choose a normal coordinate frame involving the orthonormal base $e_i(s)$, and since $\gamma$ is a geodesic we have that 
	\begin{align*}
	(g(X(\gamma(s)),\gamma'(s))'=g(\nabla_{\gamma'(s)}X(\gamma(s)),\gamma'(s))+
	g(X(\gamma(s)),\nabla_{\gamma'(s)}\gamma'(s))=\partial_N X^N(\gamma(s)),
	\end{align*}
	where we used that all Christoffel symbols vanish at $s$ and that in the chosen chart, we have $\gamma'(s)=\partial_N$. On the other hand, we have $dr=\gamma'$ which gives
	\begin{align*}
	\Gamma_y(X(r),r)(\gamma(s))=\sum_{i=1}^{N}\partial_i(X(r))\partial_i(r)=\partial_N(g(X,\gamma'))(\gamma(s))=\partial_N X^N(\gamma(s)),
	\end{align*}
	again since $\gamma$ is a geodesic and all the Christoffel symbols vanish. We proceed to prove the lemma. Bearing in mind the asymmetry of the $e_N$ and $\tilde e_N$ derivative, we define
	$$
	A=(e_N^*,-\tilde e_N^*)\otimes(e_N^*,-\tilde e_N^*)+\sum_{i=1}^{N-1}(e_i^*,\tilde e_i^*)\otimes(e_i^*,\tilde e_i^*),
	$$
	where we use the metric to produce the dual vectors $e_i^*(s):=g(e_i(s),\cdot)\in T_{\gamma(s)}^*M$.
	$A$ is obviously symmetric, and as a sum of non-negative matrices, it is non-negative itself. Since $\{e_i | 1\leq i \leq N\}$ is an orthonormal base, we have
	$$
	A|_{T_xM}=\sum_{i=1}^{N}e_i^*\otimes e_i^*=\Gamma_x
	$$
	and similarly for $y$, hence $A\in\mathcal A$. An easy computation gives
	$$
	\operatorname{tr}(AH_\psi)=\partial_{e_N\otimes -\tilde e_N}\partial_{e_N\otimes -\tilde e_N}\psi+\sum_{i=1}^{N-1}
	\partial_{e_i\otimes \tilde e_i}
	\partial_{e_i\otimes \tilde e_i} \psi.
	$$
	Now for any $1\leq i\leq N-1$, we define the geodesic variation $\gamma^i(r,s):=\exp_{\gamma(s)}(re_i(s))$. Again, since $\phi$ is non-increasing, we have that
	$$
	\psi(\exp_x(re_i(-d/2)),\exp_y(re_i(d/2)))\leq 2\phi(\frac{L(\gamma^i(r,\cdot))}{2})
	$$
	with equality if $r=0$. Similarly as above, $\gamma^i$ is an orthogonal  variation of the length minimizing geodesic $\gamma$, so the first derivative of  $L(\gamma^i(r,\cdot))$ is zero. On the other hand, using the second variation formula, we get
	\begin{align}
	\frac{\partial ^2}{\partial r^2}L(\gamma^i(r,\cdot))\bigg|_{r=0}\notag=&g(\nabla_r \frac{d}{dr}\gamma^i,\gamma
	')\bigg|^{s=d/2}_{s=-d/2}=-\int_{a}^{b} \operatorname{Rim}(e_i,\gamma ',\gamma ',e_i) ,\notag
	\end{align}
	where we used that $\gamma^i(\cdot,s)$ is a geodesic for any $s$  and that $\frac{d}{dr}\gamma^i-g(\frac{d}{dr}\gamma^i,\gamma')\gamma^i\big|_{r=0}=e_i(s)$ which is parallel along $\gamma$ by construction. Here, $\operatorname{Rim}$ denotes the Riemannian curvature tensor. Hence, we get that
	\begin{align}
	\sum_{i=1}^{N-1}\partial_{e_i\otimes\tilde e_i}\partial_{e_i\otimes\tilde e_i}\psi(x,y)&\leq \sum_{i=1}^{N-1} \phi'\bigg(\frac{d(x,y)}{2}\bigg)\frac{\partial^2}{\partial r^2}L(\gamma^i(r,\cdot))\notag \\&=-\phi'\bigg(\frac{d(x,y)}{2}\bigg)\int_{-d/2}^{d/2}\operatorname{ric}(\gamma',\gamma') \label{visc2}
	\end{align}
	Finally, we have 
	$$
	\psi(\gamma(r-d/2),\gamma(d/2-r))\leq 2\phi\bigg(\frac{d-2r}{2}\bigg)
	$$
	with equality if $r=0$. Using that $-r$ has vanishing second derivative, we get
	\begin{align}
	\partial_{e_N\otimes -\tilde e_N}\partial_{e_N \otimes -\tilde e_N}\psi(x,y)\leq 2\phi''\bigg(\frac{d(x,y)}{2}\bigg). 
	\label{visc3}
	\end{align}
	Combining (\ref{visc1}), (\ref{visc2}), and (\ref{visc3}), we get that
	\begin{align}
	\mathcal{L}(H_\psi)+(X_x+X_y)\psi\leq& \operatorname{tr}(AH_\psi)-\int_{-d/2}^{d/2}\bigg(\frac12X(\Gamma_y(r,r))-\Gamma_y(X(r),r)\bigg)\circ\gamma(s)ds
	\notag 
	\\
	\leq& 2\phi''(d/2)-\phi'(d/2)\int_{-d/2}^{d/2}\operatorname{ric}(\gamma',\gamma')ds
	\notag \\\notag &-\phi'(d/2)\int_{-d/2}^{d/2}\bigg(\frac12X(\Gamma_y(r,r))-\Gamma_y(X(r),r)\bigg)\circ\gamma(s)ds
	\\ \notag =& 2\phi''(d/2)-\phi'(d/2)\int_{-d/2}^{d/2}R(r,r)
	\\ \notag \leq& 2\phi''(d/2)-\phi'(d/2)ad,
	\end{align}
	where we used that $L$ satisfies $\operatorname{BE}(a,\infty)$ and $\Gamma(r,r)\equiv 1$. The claim follows.  \end{proof}
\subsection{A lower bound for the principal eigenvalue}
We are now in the position to prove {Theorem \ref{maintheorem1}}. The idea is, as usual, to compare the eigenfunction of $L$ with an eigenfunction of a one-dimensional model space.
\begin{theorem}
	Let $(M,L)$ be as above and $w$ be the first non-constant Neumann eigenfunction of the operator $\frac{\partial^2}{\partial s^2} -as\frac{\partial}{\partial s}$ on $[-D/2,D/2]$ with eigenvalue $\bar \lambda$. 
	Let $\lambda$ be a non-constant Neumann eigenvalue of $L$. Then we have the estimate
	$$\operatorname{Re}(\lambda)\geq\bar \lambda.$$\label{nonsymmetric.final.thm}
\end{theorem}
\begin{proof}
	One easily verifies that $w$ is the unique minimizer of the weighted energy functional
	$$
	F(\psi):=\frac{\int_{-D/2}^{D/2}e^{-a\frac{s^2}{2}}(\psi')^2}{\int_{-D/2}^{D/2}e^{-a\frac{s^2}{2}}(\psi)^2},
	$$
	among all smooth non-zero functions with zero mean. So $w$ is well-defined and smooth. It is a well-known fact that 
	$w$ can be chosen in a way such that $w(0)=0$ and such that $w$ is positive on $(0,D/2]$. The equation is invariant under the transformation $s\mapsto -s$ and $w$ changes its sign, so it follows that $w$ is odd. Because of the Neumann condition, we cannot directly apply {Theorem \ref{modulus.of.continuity.thm}}, so let $\tilde w$ be the associated solution on the interval $[-\tilde D/2, \tilde D/2]$ for $\tilde D>D$ with eigenvalue $\tilde \lambda$. The uniqueness of solutions of ordinary differential equations gives $\tilde w'(0)>0$, since otherwise $\tilde w''(0)=\tilde w'(0)=\tilde w(0)=0$, and hence $\tilde w\equiv 0$. At $\tilde D/2$, the Neumann condition implies that
	$\tilde w''(\tilde D/2)=-\tilde \lambda w(\tilde D/2)<0$, so we have that $\tilde w'>0$ on $(\tilde D/2-\epsilon,\tilde D/2)$ for some small $\epsilon>0$. Now we assume that there is an $s\in (0,\tilde D/2)$ such that $w'(s)=0$. By the considerations above, we can choose $s$ to be maximal. However, $\tilde w''(s)=-\lambda \tilde  w(s)<0$, since $w$ is positive in $(0,\tilde D/2]$, contradicting the fact that $s$ is maximal. Hence, we have that $\tilde w'|_{[0,\tilde D/2)}>0$, in particular, $\tilde w'_{[0,D/2]}>0$.\\
	Now we define the function $\tilde \phi(s,t)=Ce^{-\tilde \lambda t}\tilde w(s)$ and let $u$ be the eigenfunction of $L$ with eigenvalue $\lambda$ and define $v:=e^{-\lambda t}u$. Since $\tilde w'(0)>0$  and since the gradient of $u$ is uniformally bounded ($M$ is compact), $\phi(s,0)$ is a modulus of continuity of $\operatorname{Re}(u)$ and $\operatorname{Im}(u)$ for $C$ sufficiently large. Furthermore, $v$, and hence also $\operatorname{Re}(v)$ and $\operatorname{Im}(v)$, satisfy a heat equation with drift. $\phi$ obviously satisfies the other constraints of {Theorem \ref{modulus.of.continuity.thm}}, so we get that $\phi(\cdot,t)$ is a modulus of continuity of the real and imaginary part of $v(\cdot,t)$ for any $t\in\mathbb{R}_+$. Since $v$ is non-constant, this can only happen if $v$ has more decay than $\phi$, that is, $\operatorname{Re}(\lambda)\geq \tilde \lambda$. This proves the theorem since $\tilde \lambda \to \bar\lambda$ as  $\tilde D\searrow D$.    \end{proof}
\begin{proof}[Proof {Theorem \ref{maintheorem1}}] This follows from \cite[Proposition 3.1]{andrews}, which states that $\bar{\lambda}\geq \frac{a}{2}+\frac{\pi^2}{D^2}$. 
\end{proof}
We have $\frac12as\frac{\partial}{\partial s}(u'u')-(as\frac{\partial}{\partial s}u)'u'=-au'^2$, so the operator 
$\frac{\partial^2}{\partial s^2} -as\frac{\partial}{\partial s}$ satisfies $\operatorname{BE}(a,\infty)$. Therefore, the result in {Theorem \ref{nonsymmetric.final.thm}} is sharp. By constructing collapsing warped product manifolds similar to the previous section, one can see that the estimate is  sharp even if we restrict ourselves to an arbitrary dimension. However, as can be seen in \cite{andrews}, {Theorem \ref{maintheorem1}}  is not sharp, even in the smaller class of operators which can be written as a Bakry-Emery Laplacian. The reason is that there is no good understanding of the principal eigenvalue of the model function.

\end{document}